\documentclass[11pt]{amsart}
\usepackage{amsmath}
\usepackage{amssymb}
\usepackage{amscd}
\usepackage{color}
\def\NZQ{\mathbb}               
\def\NN{{\NZQ N}}
\def\QQ{{\NZQ Q}}
\def\ZZ{{\NZQ Z}}
\def\RR{{\NZQ R}}
\def\CC{{\NZQ C}}

\def\PP{{\NZQ P}}

\newtheorem{Theorem}{Theorem}[section]
\newtheorem{Lemma}[Theorem]{Lemma}

\newtheorem{Proposition}[Theorem]{Proposition}
\newtheorem{Remark}[Theorem]{Remark}

\newtheorem{Definition}[Theorem]{Definition}

\let\epsilon\varepsilon
\let\phi=\varphi
\let\kappa=\varkappa

\begin{document}
\title{The Minkowski equality of big divisors}
\author{Steven Dale Cutkosky}

\thanks{Partially supported by NSF grant DMS-2054394.}

\address{Steven Dale Cutkosky, Department of Mathematics,
University of Missouri, Columbia, MO 65211, USA}
\email{cutkoskys@missouri.edu}

\begin{abstract}  We give conditions characterizing equality in the Minkowski inequality for big divisors on a projective variety. Our results draw on the extensive history of research on Minkowski inequalities in algebraic geometry.
\end{abstract}


\maketitle

\section{Introduction}   

Suppose that $X$ is a projective $d$-dimensional algebraic variety over a field $k$ and $D$ is an $\RR$-Cartier divisor on $X$.
Then the volume of $D$ is
$$
{\rm vol}(D)=\lim_{n\rightarrow \infty}\frac{\dim_k\Gamma(X,\mathcal O_X(nD))}{n^d/d!}.
$$
If $D$ is nef, then the volume of $D$ is the self intersection number ${\rm vol}(D)=(D^d)$. For an arbitrary $\RR$-Cartier divisor $D$,
$$
{\rm vol}(D)=\left\{\begin{array}{cl}
\langle D^d\rangle&\mbox{ if $D$ is pseudo effective}\\
0&\mbox{ otherwise.}
\end{array}\right.
$$
Here $\langle D^d\rangle$ is the positive intersection product. The positive intersection product $\langle D^d\rangle$ is the ordinary intersection product $(D^d)$ if $D$ is nef, but these products are different in general. More generally, given pseudo effective $\RR$-Cartier divisors $D_1,\ldots,D_p$ on $X$ with $p\le d$, there is a positive intersection product  $\langle D_1\cdot\ldots\cdot D_p\rangle$  which is a linear form on $N^1(\mathcal X)^{d-p}$, where $\mathcal X$ is the limit of all birational models of $X$. We have that 
$$
{\rm vol}(D)=\langle D^p\rangle =\langle D\rangle \cdot\ldots \cdot\langle D\rangle=\langle D\rangle^d.
$$
We denote the linear forms on $N^1(\mathcal X)^{d-p}$ by $L^{d-p}(\mathcal X)$. The theory of  intersection theory and volumes which is required for this paper is reviewed in Section \ref{PrelSect}.

Suppose that  $D_1$ and  $D_2$ are pseudo effective $\RR$-Cartier divisors on $X$. We have the Minkowski inequality
$$
{\rm vol}(D_1+D_2)^{\frac{1}{d}}\ge {\rm vol}(D_1)^{\frac{1}{d}}+{\rm vol}(D_2)^{\frac{1}{d}}
$$
which follows from Theorem \ref{Ineq+} below. Further, we have the following characterization of equality in the Minkowski inequality.

\begin{Theorem}\label{Theorem22+}  Let $X$ be a  $d$-dimensional projective variety over a field $k$. For any two big $\RR$-Cartier divisors $D_1$ and $D_2$ on $X$, 
\begin{equation}\label{Neweq20+}
{\rm vol}(D_1+D_2)^{\frac{1}{d}}\ge {\rm vol}(D_1)^{\frac{1}{d}}+{\rm vol}(D_2)^{\frac {1}{d}}
\end{equation}
with equality if and only if $\langle D_1\rangle $ and $\langle D_2\rangle$ are proportional in $L^{d-1}(\mathcal X)$. 
\end{Theorem}

In the case that $D_1$ and $D_2$ are nef and big, this is proven in \cite[Theorem 2.15]{BFJ} (over an algebraically closed field of characteristic zero) and in \cite[Theorem 6.13]{C} (over an arbitrary field). In this case of nef divisors, the condition that 
$\langle L_1\rangle $ and $\langle L_2\rangle$ are proportional in $L^{d-1}(\mathcal X)$ is just that $D_1$ and $D_2$ are proportional in $N^1(X)$. 

Theorem \ref{Theorem22+} is obtained in the case that $D_1$ and $D_2$ are big and movable and $k$ is an algebraically closed field of characteristic zero in \cite[Proposition 3.7]{LX2}. In this case the condition for equality is that $D_1$ and $D_2$ are proportional in $N^1(X)$.  Theorem \ref{Theorem22+} is established in the case that $D_1$ and $D_2$ are big $\RR$-Cartier divisors and $X$ is nonsingular, over an algebraically closed field $k$ of characteristic zero in \cite[Theorem 1.6]{LX2}.  In this case, the condition for equality is that the positive parts of the $\sigma$ decompositions of $D_1$ and $D_2$ are proportional; that is, $P_{\sigma}(D_1)$ and $P_{\sigma}(D_2)$ are proportional in $N^1(X)$.

In Section \ref{SecMink},  we modify the proof sketched in \cite{LX2} of  \cite[Proposition 3.7]{LX2} to be valid over an arbitrary field. Characteristic zero is required in the proof in \cite{LX2}  as the existence of resolution of singularities is assumed and an argument using the theory of multiplier ideals is used, which  requires characteristic zero as it relies on both resolution of singularities and Kodaira vanishing.

We will write
$$
s_i=\langle D_1^i\cdot D_2^{d-i}\rangle\mbox{ for $0\le i\le d$}. 
$$

We have the following generalization of the Khovanskii-Teissier inequalities to positive intersection numbers.

\begin{Theorem} (Minkowski Inequalities)\label{Ineq+} Suppose that $X$ is a complete algebraic variety of dimension $d$ over a field $k$ and $D_1$ and $D_2$ are pseudo effective $\RR$-Cartier divisors on $X$. Then
\begin{enumerate}
\item[1)] $s_i^2\ge s_{i+1}s_{i-1}$ for $1\le i\le d-1.$
\item[2)]  $s_is_{d-i}\ge s_0s_d$ for $1\le i\le d-1$.
\item[3)] $s_i^d\ge s_0^{d-i}s_d^i$ for $0\le i\le d$.
\item[4)] ${\rm vol}(D_1+D_2) \ge {\rm vol}(D_1)^{\frac{1}{d}}+{\rm vol}(D_2)^{\frac{1}{d}}$.
\end{enumerate}
\end{Theorem}

Theorem \ref{Ineq+} follows from  \cite[Theorem 2.15]{BFJ} when $k$ has characteristic zero and from \cite[Theorem 6.6]{C} in general.
When $D_1$ and $D_2$ are nef, the inequalities of Theorem \ref{Ineq+} are proven by Khovanskii  and Teissier \cite{T1}, \cite{T2}, \cite[Example 1.6.4]{L}.  In the case that $D_1$ and $D_2$ are nef, we have that $s_i=\langle D_1^i\cdot D_2^{d-i}\rangle=(D_1^i\cdot D_2^{d-i})$ are the ordinary intersection products.

We have the following characterization of equality in these inequalities. 

\begin{Theorem} (Minkowski equalities)\label{Minkeq+} Suppose that $X$ is a projective algebraic variety of dimension $d$ over a field $k$ of characteristic zero, and $D_1$ and $D_2$ are big $\RR$-Cartier divisors on $X$.  Then the following are equivalent:
\begin{enumerate}
\item[1)] $s_i^2= s_{i+1}s_{i-1}$ for $1\le i\le d-1.$
\item[2)]  $s_is_{d-i}= s_0s_d$ for $1\le i\le d-1$.
\item[3)] $s_i^d= s_0^{d-i}s_d^i$ for $0\le i\le d$.
\item[4)] $s_{d-1}^d=s_0s_d^{d-1}$.
\item[5)] ${\rm vol}(D_1+D_2) = {\rm vol}(D_1)^{\frac{1}{d}}+{\rm vol}(D_2)^{\frac{1}{d}}$.
 \item[6)] $\langle D_1\rangle$ is proportional to $\langle D_2\rangle$ in $L^{d-1}(\mathcal X)$.
\end{enumerate}
\end{Theorem}

Theorem \ref{Minkeq+} is valid over any field $k$ when $\dim X\le 3$, since resolution of singularities is true in these dimensions.
When $D_1$ and $D_2$ are nef and big, then Theorem \ref{Minkeq+} is proven  in \cite[Theorem 2.15]{BFJ} when $k$ has characteristic zero and in  \cite[Theorem 6.13]{C} for arbitrary $k$. When $D_1$ and $D_2$ are nef and big, the condition
6) of Theorem \ref{Minkeq+} is just that $D_1$ and $D_2$ are proportional in $N^1(X)$.

The proof of Theorem \ref{Minkeq+} relies on the following Diskant inequality for big divisors. 

 Suppose that $X$ is a projective variety and $D_1$ and $D_2$ are $\RR$-Cartier divisors on $X$. The slope $s(D_1,D_2)$ of $D_2$ with respect to $D_1$ is the smallest real number $s=s(D_1,D_2)$ such that $\langle D_1\rangle \ge s\langle D_2\rangle$.
 
\begin{Theorem}\label{PropNew60+}(Diskant inequality for big divisors) 
Suppose that $X$ is a  projective $d$-dimensional variety over a field $k$ of characteristic zero and $D_1,D_2$
 are big  $\RR$-Cartier divisors on $X$.  Then 
 \begin{equation}
\langle D_1^{d-1} \cdot D_2\rangle ^{\frac{d}{d-1}}-{\rm vol}(D_1){\rm vol}(D_2)^{\frac{1}{d-1}}
\ge [\langle D_1^{d-1}\cdot D_2\rangle^{\frac{1}{d-1}}-s(D_1,D_2){\rm vol}(D_2)^{\frac{1}{d-1}}]^d.
\end{equation} 
 \end{Theorem}
 
 The Diskant inequality is proven for nef and big divisors in \cite[Theorem G]{BFJ} in characteristic zero and in  \cite[Theorem 6.9]{C} for nef and big divisors over an arbitrary field. In the case that $D_1$ and $D_2$ are nef and big, the condition  that $\langle D_1\rangle - s \langle D_2\rangle$ is pseudo effective in $L^{d-1}(\mathcal X)$ is that $D_1-sD_2$ is pseudo effective in $N^1(X)$. The Diskant inequality is proven when $D_1$ and $D_2$ are big and movable divisors and $X$ is a projective variety over an algebraically closed field of characteristic zero in \cite[Proposition 3.3, Remark 3.4]{LX2}.  Theorem \ref{PropNew60+} is a consequence of   \cite[Theorem 3.6]{DF}.

Generalizing Teissier \cite{T1}, we 
 define the inradius of $\alpha$ with respect to $\beta$ as
$$
r(\alpha;\beta)=s(\alpha,\beta)
$$
and the outradius of $\alpha$ with respect to $\beta$ as
$$
R(\alpha;\beta)=\frac{1}{s(\beta,\alpha)}.
$$

We deduce the following consequence of the Diskant inequality.

\begin{Theorem}\label{TheoremH+} Suppose that $X$ is a $d$-dimensional projective variety over a field $k$ of characteristic zero and $\alpha,\beta$ are big $\RR$-Cartier divisors on $X$. Then
\begin{equation}
\frac{s_{d-1}^{\frac{1}{d-1}}-(s_{d-1}^{\frac{d}{d-1}}-s_0^{\frac{1}{d-1}}s_d)^{\frac{1}{d}}}{s_0^{\frac{1}{d-1}}}
\le r(\alpha;\beta)\le \frac{s_d}{s_{d-1}}\le\frac{s_1}{s_0}\le R(\alpha;\beta)\le 
\frac{s_d^{\frac{1}{d-1}}}{s_1^{\frac{1}{d-1}}-(s_1^{\frac{d}{d-1}}-s_d^{\frac{1}{d-1}}s_0)^{\frac{1}{d}}}.
\end{equation}
\end{Theorem}

This gives a solution to \cite[Problem B]{T1} for big $\RR$-Cartier divisors. The inequalities of Theorem \ref{TheoremH+} are proven by Teissier in \cite[Corollary 3.2.1]{T1} for divisors on surfaces satisfying some conditions. 
In the case that $D_1$ and $D_2$ are nef and big on a projective variety over a field of characteristic zero, Theorem \ref{TheoremH+} follows from the Diskant inequality \cite[Theorem F]{BFJ}. In the case that $D_1$ and $D_2$ are nef and big on a projective variety over an arbitrary field, Theorem \ref{TheoremH+} is proven in \cite[Theorem 6.11]{C}, as a consequence of the Diskant inequality \cite[Theorem 6.9]{C} for nef divisors. 






\section{Preliminaries}\label{PrelSect}




In this section we review some properties of cycles and intersection theory on projective varieties over an arbitrary field. 

\subsection{Codimension 1 cycles} 
To establish notation we give a quick review of  some material from \cite{Kl}, \cite[Chapter 2]{F} and \cite[Chapter 1]{L}. Although the ongoing assumption in \cite{L} is that $k=\CC$, this assumption is not needed in the material reviewed in this subsection.

Let $X$ be a $d$-dimensional projective variety over a field $k$.
The group of Cartier divisors on $X$ is denoted by ${\rm Div}(X)$. There is a natural homomorphism from  ${\rm Div}(X)$ to the $(k-1)$-cycles (Weil divisors) $Z_{k-1}(X)$ of $X$ written as $D\mapsto [D]$. Further, there is a natural homomorphism
${\rm Div}(X)\rightarrow \mbox{Pic}(X)$ given by $D\mapsto \mathcal O_X(D)$. 

Denote numerical equivalence on ${\rm Div}(X)$ by $\equiv$. For $D$ a Cartier divisor, $D\equiv 0$ if and only if $(C\cdot D)_X:=\mbox{deg}(\mathcal O_X(D)\otimes\mathcal O_C)=0$ for all integral curves $C$ on $X$.

The group  $N^1(X)_{\ZZ}={\rm Div}(X)/\equiv$ and  $N^1(X)=N_1(X)_{\ZZ}\otimes \RR$. An element of  ${\rm Div}(X)\otimes\mathcal \QQ$ will be called a $\QQ$-Cartier divisor and an element of ${\rm Div}(X)\otimes \RR$  will be called an $\RR$-Cartier divisor. In an effort to keep  notation as simple as possible, the class in $N^1(X)$ of an $\RR$-Cartier divisor $D$ will often be denoted by $D$.

 We will also denote  the numerical equivalence on $Z_{d-1}(X)$ defined on  page 374 \cite{F} by $\equiv$. Let  $N_{d-1}(X)_{\ZZ}=Z_{d-1}(X)/\equiv$ and $N_{d-1}(X)=N_{d-1}(X)_{\ZZ}\otimes_{\ZZ}\RR$.
 There is a natural homomorphism $N^1(X)\rightarrow N_{d-1}(X)$ which is induced by associating to the class  of a $\RR$-Cartier divisor $D$ the class in $N_{d-1}(X)$ of its associated Weil divisor $[D]$ \cite[Section 2.1]{F}.   If $f:Y\rightarrow X$ is a morphism,
 the cycle map $f_*:Z_{d-1}(Y)\rightarrow Z_{d-1}(X)$ of \cite[Section 1.4]{F} induces a homomorphism $f_*:N_{d-1}(Y)\rightarrow N_{d-1}(X)$ (\cite[Example 19.1.6]{F}).

   Suppose that $f:Y\rightarrow X$ is a dominant morphism where $Y$ is projective variety. Then $f^*:{\rm Div}(X)\rightarrow 
  {\rm Div}(Y)$ is defined by taking local equations of $D$ on $X$ as local equations of $f^*(D)$ on $Y$. There is an induced homomorphism $f^*:N^1(X)\rightarrow N^1(Y)$ which is an injection by \cite[Lemma 1]{Kl}.
  By \cite[Proposition 2.3]{F}, we have that if $D$ is an $\RR$-Cartier divisor on $X$, then 
  \begin{equation}\label{eq41}
  f_*{[f^* D]}=\mbox{deg}(X'/X) D
  \end{equation} 
 where $\mbox{deg}(X'/X)$ is the index of the function field of $X$ in the function field of $X'$.
 
 In this subsection, we will use the notation for intersection numbers of \cite[Definition 2.4.2]{F}.
 
 The first statement of the following lemma follows immediately from \cite{M} or \cite[Corollary XIII.7.4]{Kl2} if $k$ is algebraically closed. The second statement is \cite[Example 19.1.5]{F}.
   
 \begin{Lemma}\label{Lemma55} Let $X$ be a $d$-dimensional projective variety over a field $k$. Then: 
 \begin{enumerate}
 \item[1)] The homomorphism $N^1(X)\rightarrow N_{d-1}(X)$ is an injection.
 \item[2)] If $X$ is nonsingular, then the homomorphism $N^1(X)\rightarrow N_{d-1}(X)$ is an isomorphism.
 \end{enumerate}
 \end{Lemma} 
 
 \begin{proof}  Suppose that $N^1(X)\rightarrow N_{d-1}(X)$ is not injective.
 The homomorphism $N^1(X)\rightarrow N_{d-1}(X)$ is obtained by tensoring the natural map $N_1(X)_{\ZZ}\otimes_{\ZZ}\QQ\rightarrow N_{d-1}(X)_{\ZZ}\otimes_{\ZZ}\QQ$ with $\RR$ over $\QQ$. Thus $N_1(X)_{\ZZ}\otimes_{\ZZ}\QQ\rightarrow N_{d-1}(X)_{\ZZ}\otimes_{\ZZ}\QQ$ is not injective, and so there exists a Cartier divisor $D$ on $X$ 
 such  that the Weil divisor $[D]$ associated to $D$ is numerically equivalent to zero (its class is zero in $N_{d-1}(X)$) but the class of  $D$ is not zero  in $N^1(X)$. Thus there exists  
   an integral curve $C$ on $X$ such that 
   \begin{equation}\label{eq59}
   (C\cdot D)_X\ne 0.
   \end{equation}
    Let $\overline k$ be an algebraic closure of $k$. There exists an integral subscheme $\overline X$  of $X\otimes_k\overline k$ such that $\overline X$ dominates $X$. Thus $\overline X$ is a projective variety over $\overline k$. Let $\psi:\overline X\rightarrow X$ be the induced dominant morphism. Let $U\subset X$ be an affine open subset such that $U\cap C\ne\emptyset$.  $\psi^{-1}(U)$ is affine since it is a closed subscheme of the affine scheme $U\otimes_k\overline k$. Let $A=\Gamma(U,\mathcal O_X)$ and $B=\Gamma(\psi^{-1}(U),\mathcal O_{\overline X})$. The ring extension $A\rightarrow B$ is integral. Let $P=\Gamma(U,\mathcal I_C)$, a  prime ideal of $A$ such that $\dim A/P=1$, and let $M$ be a maximal ideal of $A$ containing $P$.  By the going up theorem, there exists a prime ideal $Q$ of $B$ such that $Q\cap A=P$ and prime ideal $N$ of $B$ such that $Q\subset N$ and $N\cap A=M$. Now $A/M\rightarrow B/N$ is an integral extension from a field to a domain, so $B/N$ is a field. Thus $N$ is a maximal ideal of $B$ and since there are no prime ideals of $B$ properly between $Q$ and $N$ (by \cite[Corollary 5.9]{At}) we have that $\dim B/Q=1$. Let $\overline C$ be the closure of $V(Q)\subset \psi^{-1}(U)$ in $\overline X$. Then $\overline C$ is an integral curve on $X$ which dominates $C$. There exists a field of definition $k'$ of $\overline X$ and $\overline C$ over $k$ which is a subfield of $\overline k$ which is finite over $k$. That is, there exist subvarieties $C'\subset X'$ of $X\otimes_kk'$ such that $X'\otimes_{k'}\overline k=\overline X$ and $C'\otimes_{k'}\overline k=\overline C$.  We factor $\psi:\overline X\rightarrow X$ by morphisms
   $$
   \overline X\stackrel{\alpha}{\rightarrow} X'\stackrel{\phi}{\rightarrow} X
   $$
   where $\alpha={\rm id}_X'\otimes_{{\rm id}_{k'}}{\rm id}_{\overline k}$.  The morphism $\phi$ is finite and surjective and $\alpha$ is flat (although it might not be of finite type). Let $H$ be an ample Cartier divisor on $X$.
   Then $\phi^*H$ is an ample Cartier divisor on  $X'$ (by \cite[Exercise III.5.7(d)]{H}). Thus for some positive integer $m$ we have that global sections of $\mathcal O_{X'}(m\phi^*(H))$  give a closed embedding of $X'$ in $\PP^n_{k'}$ for some $n$. Thus 
 global sections of $\mathcal O_{\overline X}(m\psi^*(H))$  give a closed embedding of $\overline X=X'\otimes_{k'}\overline k$ in $\PP^n_{\overline k}$. In particular, we have that $\psi^*(H)$ is an ample Cartier divisor on $\overline X$. We have natural morphisms
 $$
 N^1(X)\rightarrow N^1(X')\rightarrow N^1(\overline X).
 $$
 Here $X$ is a $k$-variety and $\overline X$ is a $\overline k$-variety. $X'$ is both a $k$-variety and a $k'$-variety. When we are regarding $X'$ as a $k$-variety we will write $X'_k$ and when we are regarding $X'$ as a $k'$-variety we will write $X'_{k'}$.
 
 We may use the formalism of   Kleiman \cite{Kl}, using the Snapper polynomials \cite{Sn} to compute intersection products of Cartier divisors. This is consistent with the intersection products of Fulton \cite{F} by \cite[Example 18.3.6]{F}.  This intersection theory is also presented in \cite[Chapter 19]{AG}.
 
 Since $D$ is numerically equivalent to zero as a Weil divisor, we have that
 \begin{equation}\label{eq56}
 (D\cdot H^{d-1})_X=(D^2\cdot H^{d-2})_X=0.
 \end{equation}
 We have that 
 $$
 (\psi^*D\cdot \psi^*H^{d-1})_{\overline X}=(\phi^*D\cdot \phi^*H^{d-1})_{X'_{k'}}=\frac{1}{[k':k]}(\phi^* D\cdot \phi^* H^{d-1})_{X'_k}
 $$
  using \cite[Example 18.3.6]{F} and the fact that 
 $$
 H^i(\overline X,\mathcal O_{\overline X}(\psi^*(mD)+\psi^*(nH)))=H^i(X'_{k'},\mathcal O_{X'}(\phi^*(mD)+\phi^*(nH)))\otimes_{k'}\overline k
 $$
 for all $m,n$ since $\alpha$ is flat. We thus have that 
 \begin{equation}\label{eq57}
 (\psi^*D\cdot \psi^*H^{d-1})_{\overline X}=\frac{1}{[k':k]}(\phi^* D\cdot \phi^* H^{d-1})_{X'_k}=\frac{\deg(X'/X)}{[k':k]}(D\cdot H^{d-1})_X=0
 \end{equation}
 by  \cite[Proposition 2.3]{F} and (\ref{eq56}). Similarly, 
 \begin{equation}\label{eq58}
 (\psi^*D^2\cdot \psi^*H^{d-2})_{\overline X}=0.
 \end{equation}
 
 Since  $\overline k$ is algebraically closed and the equations (\ref{eq57}) and (\ref{eq58}) hold,  we have that 
 $$
 (\psi^*D\cdot \overline C)_{\overline X}=0
 $$
  by  \cite{M} and \cite[Corollary XIII.7.4]{Kl2}. Thus by \cite[Example 18.3.6 and Proposition 2.3]{F},
  $$
  \begin{array}{lll}
  0&=&(\psi^*D\cdot \overline C)_{\overline X}=(\phi^*D\cdot C')_{X'_{k'}}
  =\frac{1}{[k':k]}(\phi^*D\cdot C')_{X'_{k}}\\
  &=&\frac{1}{[k':k]}(D\cdot \phi_*C')_X=\frac{\deg(C'/C)}{[k':k]}(D\cdot C)_X,
  \end{array}
  $$
  giving a contradiction to (\ref{eq59}). Thus the map $N^1(X)\rightarrow N_{d-1}(X)$ is injective.
   
 This homomorphism is always an  isomorphism if $X$ is nonsingular by \cite[Example 19.1.5]{F}. 
  \end{proof}

 As defined and developed in \cite{Kl}, \cite[Chapter 2]{L}, there are important cones ${\rm Amp}(X)$ (the ample cone), ${\rm Big}(X)$ (the big cone), ${\rm Nef}(X)$ (the nef cone) and ${\rm Psef}(X):=\overline{\rm Eff}(X)$ (the pseudo effective cone) in $N^1(X)$.
  
   If $D$ is a Cartier divisor on the projective variety $X$, then the  complete linear system $|D|$ is defined by
   \begin{equation}\label{eq30}
   |D|=\{{\rm div}(\sigma)\mid \sigma\in \Gamma(X,\mathcal O_X(D))\}.
   \end{equation}
   Let ${\rm Mov'}(X)$ be the convex cone in $N^1(X)$ generated by the classes of Cartier divisors $D$ such that $|D|$ has no codimension 1 fixed component. Define $\overline {\rm Mov}(X)$ to be the closure of ${\rm Mov'}(X)$ in $N^1(X)$. An $\RR$-Cartier divisor $D$ is said to be movable if the class of $D$ is in $\overline {\rm Mov}(X)$.   Define ${\rm Mov}(X)$ to be the interior of $\overline{\rm Mov}(X)$. As explained in \cite[page 85]{N}, we have  inclusions
   $$
   {\rm Amp}(X)\subset {\rm Mov}(X)\subset {\rm Big}(X)
   $$
   and
   $$
   {\rm Nef}(X)\subset \overline{\rm Mov}(X)\subset {\rm Psef}(X).
   $$

\begin{Lemma}\label{Lemma7} Suppose that $X$ is a $d$-dimensional  variety over a field $k$, $D$ is a pseudo effective $\RR$-Cartier divisor on $X$, $H$ is an ample $\QQ$-Cartier divisor on $X$ and $(H^{n-1}\cdot D)_X=0$. Then $D\equiv 0$.
\end{Lemma}

\begin{proof} We will  establish the lemma when $k$ is algebraically closed. The lemma will then follow for arbitrary $k$ by the method of the proof of Lemma \ref{Lemma55}.

We consider two operations on varieties. First suppose that $Y$ is a projective variety of dimension $d\ge 2$ over $k$, $\tilde H$ is an ample $\QQ$-Cartier divisor  and $\tilde D$ is a pseudo effective  $\RR$-Cartier divisor on $Y$ and $\tilde C$ is an integral curve on $Y$. Let 
$\pi:\overline Y\rightarrow Y$ be the normalization of $Y$. Then there exists an integral curve $\overline C$ in $\overline Y$ such that $\pi(\overline C) =\tilde C$ (as in the proof of Lemma \ref{Lemma55}). We have that
$$
(\pi^*(\tilde H)^{d-1}\cdot \pi^*(\tilde D))_{\overline Y}=(\tilde H^{d-1}\cdot \tilde D)_Y
$$
 and  
$$
(\overline C\cdot \pi^*(\tilde D))_{\overline Y}=\mbox{deg}(\overline C/\tilde C)(\tilde C\cdot \tilde D)_Y.
$$
 We further have that $\pi^*(\tilde D)$ is pseudo effective.

For the second operation, suppose that $Y$ is a normal projective variety over $k$. Let $\tilde H$ be an ample $\QQ$-Cartier divisor on $Y$ and $\tilde D$ be a pseudo effective $\RR$-Cartier divisor on $Y$. Let
$\tilde C$ be an integral curve on $Y$. Let $\phi:Z:=B(\tilde C)\rightarrow Y$ be the blow up of $\tilde C$. Let $E$ be the effective Cartier divisor on $Z$ such that $\mathcal O_Z(-E)=\mathcal I_{\tilde C}\mathcal O_Z$. There exists a positive integer $m$ such that $m\tilde H$ is a Cartier divisor and $\phi^*(m\tilde H)-E$ is very ample on $Z$. Let $L$ be the linear system
$$
L=\{F\in |mH|\mid \tilde C\subset \mbox{Supp}(F)\}
$$
on $Y$. The base locus of $L$ is $\tilde C$. We have an induced rational map $\Phi_L:X\dashrightarrow \PP^n$ where $n$ is the dimension of $L$. Let $Y'$ be the image of $\Phi_L$. Then $Y'\cong Z$ since $\phi^*(m\tilde H)-E$ is very ample on $Z$. Thus $\dim Y'=d$ and we have equality of  function fields $k(Y')=k(Y)$. By the first theorem of Bertini, \cite{M1}, \cite[Section I.7]{Z}, \cite[Theorem 22.12]{AG}, a general member $W$ of $L$ is  integral, so that it is a variety. By construction, $\tilde C\subset W$. Let $\alpha:W\rightarrow Y$ be the inclusion. We have that $\alpha^*(\tilde H)$ is ample on $W$. A general member of $L$ is not a component  of the support of $\tilde D$ so $\alpha^*(\tilde D)$ is pseudo effective. We have that
$(\alpha^*(\tilde H)^{d-2}\cdot \alpha^*(\tilde D))_W=(\tilde H^{d-1}\cdot \tilde D)_Y$.
  Further,
$(\tilde C\cdot \alpha^*(\tilde D))_W=(\tilde C\cdot \tilde D)_Y$.

Suppose that $D$ is not numerically equivalent to zero. We will derive a contradiction. There then exists an integral curve $C$ on $X$ such that $(C\cdot D)_X\ne 0$.
By iterating the above two operations, we construct a morphism of $k$-varieties $\beta:S\rightarrow X$ such that $S$ is a two dimensional projective variety, with an integral curve 
$\tilde C$ on $S$, an ample $\QQ$-Cartier divisor $\tilde H$ on $S$ and a pseudo effective $\RR$-Cartier divisor on $S$ such that
$(\tilde H\cdot \tilde D)_S=0$ but $(\tilde D\cdot \tilde C)_S\ne 0$. Let $\gamma:T\rightarrow S$ be a resolution of singularities (which exists by \cite{Ab}, \cite{Li} or \cite{CJS}). 
There exists an exceptional divisor $E$ on $T$ and a positive integer $m$ such that $m\tilde H$ is a Cartier divisor on $S$ and $A:=\gamma^*(m\tilde H)-E$ is an ample $\QQ$-Cartier divisor. There exists an integral curve $\overline C$ on $T$ such that $\gamma(\overline C)=\tilde C$ and $\gamma^*(\tilde D)$ is a pseudo effective $\RR$-Cartier divisor. Since $E$ is exceptional for $\gamma$, We have that
$$
(A\cdot \gamma^*(\tilde D))_T=(\gamma^*(m\tilde H)-E)\cdot \gamma^*(\tilde D))_T=(\gamma^*(m\tilde H)\cdot \gamma^*(\tilde D))_T=m(\tilde H\cdot \tilde D)_S=0
$$
and
$$
(\gamma^*(\tilde D)\cdot \overline C)=\deg(\overline C/\tilde C)(\tilde C\cdot \tilde D)_S\ne 0
$$
by \cite[Chapter I]{Kl}, \cite[Proposition 19.8 and Proposition 19.12]{AG}. But this is a contradiction to  \cite[Theorem 1, page 317]{Kl}, \cite[Theorem 1.4.29]{L}, since $N^1(T)=N_1(T)$ by Lemma \ref{Lemma55}.

\end{proof}

\subsection{Normal varieties}\label{subsecnorm} In this section we review some material from \cite{FKL}.
Suppose that $X$ is a normal projective variety over a field $k$. The map $D\rightarrow [ D]$ is an inclusion of  ${\rm Div}(X)$ into  $Z_{d-1}(X)$, and thus induces an  inclusion of  ${\rm Div}(X)\otimes\mathcal \RR$ into $Z_{d-1}(X)\otimes \RR$. We may thus identify a Cartier divisor $D$ on $X$ with its associated Weil divisor $[D]$.

Let $x$ be a real number.  Define $\lfloor x\rfloor$ to be  the round down of $x$ and $\{x\}=x-\lfloor x\rfloor$. 
Let $E$ be an $\RR$-Weil divisor on a normal variety $X$ (an element of $Z_{d-1}(X)\otimes \RR$).  Expand $E=\sum a_iE_i$ with $a_i\in \RR$ and $E_i$ prime divisors on $X$. Then we have associated divisors
$$
\lfloor E\rfloor =\sum \lfloor a_i\rfloor E_i\mbox{ and }\{E\}=\sum \{a_i\}E_i.
$$ 
There is an associated sheaf coherent sheaf $\mathcal O_X(E)$ on $X$ defined by 
$$
\Gamma(U,E)=\{f\in k(X)^*\mid \mbox{div}(f)+E|_U \ge 0\}\mbox{ for $U$ an open subset of $X$.}
$$
 We have that $\mathcal O_X(D)=\mathcal O_X(\lfloor D\rfloor)$. 
If $D$ and $D'$ are $\RR$-Weil divisors on $X$, then define $D'\sim_{\ZZ}D$ if $D'-D=\mbox{div}(f)$ for some $f\in k(X)$. 
Define $D'\sim_{\QQ}D$ if there exists $m\in \ZZ_{>0}$ such that $mD'\sim_{\ZZ}mD$.


For $D$ an $\RR$-Weil divisor, the complete linear system $|D|$ is defined as  
$$
|D|=\{\mbox{$\RR$-Weil divisors }D'\mid D'\ge 0\mbox{ and }D'\sim_{\ZZ}D\}.
$$
If $D$ is an integral Cartier divisor, then this is in agreement with the definition of (\ref{eq30}). For $D$ an $\RR$ Weil divisor, we define
$$
|D|_{\QQ}=\{\mbox{$\RR$-Weil divisors } D'\mid D'\ge 0\mbox{ and }D'\sim_{\QQ}D\}.
$$

\subsection{$\sigma$-decomposition}\label{Subsecsigma} In this subsection we assume that $X$ is a nonsingular projective variety over a field $k$. We will restrict our use of $\sigma$-decompositions to this situation. Nakayama defined and developed $\sigma$-decompositions for nonsingular complex  projective varieties in Chapter III of \cite{N}. The theory and proofs in this chapter extend to arbitrary fields. 
The $\sigma$-decomposition is  extended to complete normal projective varieties in \cite{FKL}.

Since $X$ is nonsingular, the map $D\rightarrow [D]$ is an isomorphism from ${\rm Div}(X)$ to  $Z_{d-1}(X)$, and thus induces an  isomorphism ${\rm Div}(X)\otimes\mathcal \RR\rightarrow Z_{d-1}(X)\otimes \RR$. Thus we may identify  $\RR$-Cartier divisors and $\RR$-Weil divisors on  $X$, which we will refer to as $\RR$-divisors. Since $X$ is normal, we may use the theory of Subsection \ref{subsecnorm}. 

Let $D$ be an $\RR$-divisor. We define 
$$
|D|_{\rm num}=\{\mbox{$\RR$ divisors $D'$ on $X$}\mid D'\ge 0\mbox{ and }D'\equiv D\}.
$$

Let $D$ be a big $\RR$-divisor and $\Gamma$ be a prime  divisor on $X$. Then we define
$$
\sigma_{\Gamma}(D)_{\ZZ}:=\left\{\begin{array}{ll}
\inf\{\mbox{mult}_{\Gamma}\Delta\mid \Delta\in |D|\}&\mbox{ if }|D|\ne 0\\
+\infty&\mbox{ if }|D|=\emptyset,
\end{array}\right.
$$
$$
\sigma_{\Gamma}(D)_{\QQ}:=\inf\{\mbox{mult}_{\Gamma}\Delta\mid \Delta\in |D|_{\QQ}\},
$$
$$
\sigma_{\Gamma}(D):=\inf\{\mbox{mult}_{\Gamma}\Delta\mid \Delta\in |D|_{\rm num}\}.
$$
These three functions $\sigma_{\Gamma}(D)_*$ satisfy 
$$
\sigma_{\Gamma}(D_1+D_2)_*\le \sigma_{\Gamma}(D_1)_*+\sigma_{\Gamma}(D_2)_*.
$$

We have that 
\begin{equation}\label{eq37}
\sigma_{\Gamma}(D)_{\QQ}=\sigma_{\Gamma}(D)
\end{equation}
 by \cite[Lemma III.1.4]{N}. 
 
 The function $\sigma_{\Gamma}$ is continuous on ${\rm Big}(X)$ by \cite[Lemma 1.7]{N}.



If $D$ is a pseudo effective $\RR$-divisor and $\Gamma$ is a prime divisor, then 
$$
\sigma_{\Gamma}(D):=\lim_{t\rightarrow 0^+}\sigma_{\Gamma}(D+tA)
$$ 
where $A$ is any ample $\RR$-divisor on $X$. These limits exist and converge to the same number by \cite[Lemma 1.5]{N}.
By \cite[Corollary 1.11]{N}, there are only finitely many prime divisors $\Gamma$ on $X$ such that $\sigma_{\Gamma}(D)>0$.
 For a given pseudo effective $\RR$-divisor $D$, the  $\RR$-divisors 
$$
N_{\sigma}(D)=\sum_{\Gamma}\sigma_{\Gamma}(D)\Gamma
\mbox{ and }
P_{\sigma}(D)=D-N_{\sigma}(D)
$$
are defined  in \cite[Definition 1.12]{N}. The decomposition $D=P_{\sigma}(D)+N_{\sigma}(D)$ is called the $\sigma$-decomposition of $D$.


Suppose that $D$ is a  pseudo effective $\RR$-divisor, $A$ and $H$ are ample $\RR$-divisors and $t,\epsilon>0$. Then, since $D+tA+\epsilon H$, $D+\epsilon H$  and $tA$ are big, we have that  for any prime divisor $\Gamma$,
$$
\sigma_{\Gamma}(D+tA+\epsilon H)\le \sigma_{\Gamma}(D+\epsilon H)+\sigma_{\Gamma}(tA)=\sigma_{\Gamma}(D+\epsilon H).
$$
Thus
$$
\sigma_{\Gamma}(D+tA)=\lim_{\epsilon\rightarrow 0^+}\sigma_{\Gamma}(D+tA+\epsilon H)
\le \lim_{\epsilon\rightarrow 0^+}\sigma_{\Gamma}(D+\epsilon H)=\sigma_{\Gamma}(D).
$$
In particular, if $\Gamma_1,\ldots,\Gamma_s$ are the prime divisors such that $N_{\sigma}(D)=\sum_{i=1}^s a_i\Gamma_i$
where $a_i>0$ for all $i$, then for all $t>0$, there is an expansion $N_{\sigma}(D+tA)=\sum_{i=1}^sa_i(t)\Gamma_i$ where 
$a_i(t)\in \RR_{\ge 0}$. Thus $\lim_{t\rightarrow 0^+}N_{\sigma}(D+tA)=N_{\sigma}(D)$ and $\lim_{t\rightarrow 0^+}P_{\sigma}(D+tA)=P_{\sigma}(D)$.

\begin{Lemma}\label{Lemma31} Suppose that $D$ is a pseudo effective $\RR$-divisor on a nonsingular projective variety $X$. Then
\begin{enumerate}
\item[1)] $P_{\sigma}(D)$ is pseudo effective.\item[2)] $\sigma_{\Gamma}(P_{\sigma}(D))=0$ for all prime divisors $\Gamma$ on $X$, so that the class of $P_{\sigma}(D)$ is in $\overline{\rm Mov}(X)$.
\item[3)] $N_{\sigma}(D)=0$ if and only if the class of $D$ is in $\overline{\rm Mov}(X)$.
\end{enumerate}
\end{Lemma}

\begin{proof} 
Let $A$ be an ample $\RR$-divisor on $X$. For all $\epsilon>0$, $D+\epsilon A$ is big. Thus the class of 
$D+\epsilon A-\sum \sigma_{\Gamma}(D+\epsilon A)\Gamma$ is in $\mbox{Big}(X)$.
 Thus $P_{\sigma}(D)=\lim_{\epsilon \rightarrow 0^+}D+\epsilon A-\sum \sigma_{\Gamma}(D+\epsilon A)\Gamma$ is pseudo effective. 
Statement 2) follows from \cite[Lemma III.1.8]{N} and \cite[Proposition III.1.14]{N}. Statement 3) is \cite[Proposition III.1.14]{N}.

\end{proof}


\subsection{Movable divisors on a normal variety}\label{SubSecMN}

Let $X$ be a normal projective variety over a field, and $\Gamma$ be a prime divisor on $X$. As explained in \cite{FKL},
 the definitions of $\sigma_{\Gamma}(D)_{\ZZ}$ and $\sigma_{\Gamma}(D)_{\QQ}$ of Subsection \ref{Subsecsigma} extend to $\RR$-Weil divisors $D$ on $X$, as do the inequalities
 $$
 \sigma_{\Gamma}(D_1+D_2)_{\ZZ}\le \sigma_{\Gamma}(D_1)_{\ZZ}+\sigma_{\Gamma}(D_2)_{\ZZ}
 \mbox{ and }
 \sigma_{\Gamma}(D_1+D_2)_{\QQ}\le \sigma_{\Gamma}(D_1)_{\ZZ}+\sigma_{\Gamma}(D_2)_{\QQ}.
 $$
 Let $D$ be a big and movable $\RR$-Cartier divisor on $X$ and $A$ be an ample $\RR$-Cartier divisor on $X$. Then $D+tA\in {\rm Mov}(X)$ for all positive $t$, so that $\sigma(D+tA)_{\QQ}=0$ for all $t>0$. Since $D$ is big, there exists $\delta>0$ such that $D\sim_{\QQ}\delta A+\Delta$ where $\Delta$ is an effective $\RR$-Cartier divisor. Then for all $\epsilon >0$, $(1+\epsilon)D\sim_{\QQ}D+\epsilon\delta A+\epsilon\Delta$ and so 
 $$
 (1+\epsilon)\sigma_{\Gamma}(D)_{\QQ}\le \sigma(D+\epsilon\delta A)_{\QQ}+\epsilon\mbox{mult}_{\Gamma}(\Delta)=\epsilon\mbox{mult}_{\Gamma}(\Delta)
 $$
 for all $\epsilon>0$. Thus, with our assumption that $D$ is a big and movable $\RR$-Cartier divisor, we have that
 \begin{equation}\label{eq50}
 \sigma_{\Gamma}(D)_{\QQ}=0\mbox{ for all prime divisors $\Gamma$ on $X$.}
 \end{equation}

\subsection{Positive intersection products} Let $X$ be a $d$-dimensional projective variety over a field $k$. In \cite{C}, we generalize the positive intersection product on  projective varieties over an algebraically closed field of characteristic zero defined in \cite{BFJ} to projective varieties over an arbitrary field.

 Let $I(X)$ be the directed set of projective varieties $Y$ which have a birational morphism to $X$. If $f:Y'\rightarrow Y$ is in $I(X)$ and $\mathcal L\in N^1(Y)$, then $f^*\mathcal L\in N^1(Y')$. We may thus define $N^1(\mathcal X)=\lim_{\rightarrow }N^1(Y)$. If $D$ is a Cartier $\RR$-divisor on $Y$, we will sometimes abuse notation and identify $D$ with is class in $N^1(\mathcal X)$.
In \cite{C}, $N^1(Y)$ is denoted by $M^1(Y)$.

 For $Y\in I(X)$ and $0\le p\le d$, we let $L^p(Y)$ be the real vector space of $p$-multilinear forms on $N^1(Y)$.
  Giving the finite dimensional real vector space $L^p(Y)$ the Euclidean topology, we define 
 \begin{equation}\label{eq300}
 L^p(\mathcal X)=\lim_{\leftarrow}L^p(Y).
 \end{equation}
  $L^p(\mathcal X)$ is a Hausdorff topological real vector space. We define $L^0(\mathcal X)=\RR$.
The pseudo effective cone ${\rm Psef}(L^p(Y))$ in $L^p(Y)$ is the closure of the cone generated by the natural images of the $p$-dimensional closed subvarieties of $Y$. The inverse limit of the ${\rm Psef}(L^p(Y))$ is then a closed convex and strict cone ${\rm Psef}(L^p(\mathcal X))$ in $L^p(\mathcal X)$, defining a partial order $\ge$ in $L^p(\mathcal X)$. The pseudo effective cone in $L^0(\mathcal X)$ is the set of nonnegative real numbers.
For $Y\in I(X)$, let $\rho_Y:N^1(Y)\rightarrow N^1(\mathcal X)$ and  $\pi_Y:L^p(\mathcal X)\rightarrow L^p(Y)$ be the induced continuous linear maps.  In  \cite{BFJ} they consider a related but different vector space from $L^p(\mathcal X)$.

Suppose that $\alpha_1,\ldots,\alpha_r\in N^1(\mathcal X)$ with $r\le d$. Let $f:Y\rightarrow X\in I(X)$  be such that $\alpha_1,\ldots,\alpha_r$
are represented by classes in $N^1(Y)$ of $\RR$-Cartier divisors $D_1,\ldots,D_r$ on $Y$. Then the ordinary intersection product $D_1\cdot\ldots\cdot D_r$ induces a linear map  $D_1\cdot\ldots\cdot D_r\in L^{d-r}(\mathcal X)$. If $r=d$, then this linear map is just the intersection number $(D_1\cdot\ldots\cdot D_d)_Y\in \RR$ of \cite[Definition 2.4.2]{F}.


If $\alpha_1,\ldots,\alpha_p \in N^1(\mathcal X)$ are big, we define the positive intersection product (\cite[Definition 2.5, Proposition 2.13]{BFJ} in characteristic zero, \cite[Definition 4.4, Proposition 4.12]{C}) to be 
\begin{equation}\label{eq33}
\begin{array}{llll}
\langle \alpha_1\cdot \ldots\cdot\alpha_p\rangle &=& {\rm lub}&\{(\alpha_1- D_1)\cdots (\alpha_p- D_p)\in L^{d-p}(\mathcal X)\mid D_i \mbox{ are effective $\RR$-Cartier}\\
&&&\mbox{divisors  on some $Y_i\in I(X)$  and $\alpha- D_i$ are big}\}.
\end{array}
\end{equation}

\begin{Proposition}\label{Prop35}(\cite[Proposition 2.13]{BFJ}, \cite[Proposition 4.12]{C})
 If $\alpha_1,\ldots,\alpha_p\in N^1(\mathcal X)$ are big, we have that 
$\langle \alpha_1\cdot \ldots\cdot \alpha_p\rangle$ is the least upper bound in $L^{d-p}(\mathcal X)$ of all intersection products
$\beta_1\cdot\ldots\cdot\beta_p$ where $\beta_i$ is the class of a nef $\RR$-Cartier divisor such that $\beta_i\le \alpha_i$ for all $i$.
\end{Proposition}

If $\alpha_1,\ldots,\alpha_p\in N^1(\mathcal X)$ are pseudo effective, their positive intersection product is defined  (\cite[Definition 2.10]{BFJ}, \cite[Definition 4.8, Lemma 4.9]{C}) as 
$$
\lim_{\epsilon\rightarrow 0^+}\langle (\alpha_1+\epsilon H)\cdot \ldots \cdot (\alpha_p+\epsilon  H)\rangle
$$
where $H$ is a big $\RR$-Cartier divisor on some $Y\in I(X)$.  

\begin{Lemma}\label{Lemma36}(\cite[Proposition 2.9, Remark 2.11]{BFJ}, [\cite[Lemma 4.13]{C}, \cite[Proposition 4.7]{C}) The positive intersection product $\langle \alpha_1\cdot\ldots\cdot\alpha_p\rangle$ is homogeneous and super additive on each  variable in ${\rm Psef}(\mathcal X)$. Further, it is continuous on the $p$-fold product of the big cone.  \end{Lemma} 

\begin{Remark}\label{Remark50}
Since a positive intersection product is always in the pseudo effective cone, if $\alpha_1,\ldots,\alpha_d\in N^1(\mathcal X)$ are pseudo effective, then 
$\langle \alpha_1\cdot\ldots\cdot\alpha_d\rangle\in \RR_{\ge 0}$. Since the intersection product of nef and big $\RR$-Cartier divisors is positive, it follows from Proposition \ref{Prop35} that if  $\alpha_1,\ldots,\alpha_d\in N^1(\mathcal X)$ are big, then
$\langle \alpha_1\cdot\ldots\cdot\alpha_d\rangle\in \RR_{> 0}$.
\end{Remark}

\begin{Lemma}\label{Lemma34} Let $H$ be an ample $\RR$-Cartier divisor on some $Y\in I(X)$ and let $\alpha\in N^1(\mathcal X)$ be pseudo effective. Then 
$$
\langle  H^{d-1}\cdot \alpha\rangle =  H^{d-1}\cdot\langle \alpha \rangle.
$$
\end{Lemma}

\begin{proof} By Proposition \ref{Prop35}, for all $\epsilon >0$,$$
\langle \left((1+\epsilon) H)\right)^{d-1}\cdot(\alpha+\epsilon  H)\rangle=(1+\epsilon)^{d-1}\left( H^{d-1}\cdot \langle \alpha+\epsilon  H\rangle\right).
$$
Taking the limit as $\epsilon$ goes to zero, we have the conclusions of the lemma.
\end{proof}

\begin{Theorem}\label{Theorem17}
Suppose that $X$ is a $d$-dimensional projective variety, $\alpha\in N^1(X)$ is big and $\gamma\in N^1(X)$ is arbitrary. Then
$$
\frac{d}{dt}{\rm vol}(\alpha+t\gamma)=d\langle (\alpha+t\gamma)^{d-1}\rangle\cdot\gamma
$$
whenever $\alpha+t\gamma$ is big.
\end{Theorem}

This is a restatement of \cite[Theorem A]{BFJ}, \cite[Theorem 5.6]{C}. The proof shows that
$$
\lim_{\Delta t\rightarrow 0}\frac{{\rm vol}(\alpha+(t+\Delta t)\gamma)-{\rm vol}(\alpha+t\gamma)}{\Delta t}
=d\langle (\alpha+t\gamma)^{d-1}\rangle\cdot\gamma.
$$

Suppose $\alpha\in N^1(\mathcal X)$ is pseudo effective. Then we have for varieties over arbitrary fields, the formula of \cite[Corollary 3.6]{BFJ},
 \begin{equation}\label{eq40}
 \langle \alpha^d\rangle =\langle \alpha^{d-1}\rangle\cdot \alpha.
 \end{equation}
To establish this formula, first 
suppose that $\alpha$ is big. 
 Then taking the derivative at $t=0$ of $\langle(\alpha+t\alpha)^d\rangle=(1+t)^d\langle\alpha^d\rangle$,  we obtain formula (\ref{eq40}) from Theorem \ref{Theorem17}. If $\alpha$ is pseudo effective, we obtain (\ref{eq40}) by regarding  $\alpha$ as a limit of the big divisors $\alpha +tH$ where $H$ is an ample $\RR$-Cartier divisor. 
 
 The natural map $N^1(X)\rightarrow L^{d-1}(X)$ is an injection, as follows from the proof of Lemma \ref{Lemma55}. Let $WN^1(X)$ be the image of the homomorphism of $Z_{d-1}(X)\otimes\RR$ to $L^{d-1}(X)$ which associates to $D\in Z_{d-1}(X)\otimes\RR$ the natural map $(\mathcal L_1,\ldots,\mathcal L_{d-1})\mapsto (\mathcal L_1\cdot\ldots\cdot L_{d-1}\cdot D)_X$.
We have that  $WN^1(X)$ is the subspace of $L^{d-1}(X)$ generated by $\mbox{Psef}(X)$. We always have a factorization $N^1(X)\rightarrow N_{d-1}(X)\rightarrow WN^1(X)$.
    In this way we can identify the map $D\cdot$ which is the image  of an element of $Z_{d-1}(X)\otimes\RR$ in $L^{d-1}(X)$ with its class in $WN^1(X)$. If $X$ is nonsingular, then $WN^1(X)=N_{d-1}(X)=N^1(X)$.

 \begin{Lemma}\label{Lemma200} Suppose that $X$ is a projective variety and $D$ is a big $\RR$-Cartier divisor on $X$. Let $f:Y\rightarrow X\in I(X)$ be such that $Y$ is normal. Then
 \begin{equation}\label{eq91}
 \pi_Y(\langle D\rangle)=P_{\sigma}(f^*(D)).
 \end{equation}
 \end{Lemma}
 
 \begin{proof} We may assume that $Y=X$ so that $f^*D=D$.
After replacing $D$ with an $\RR$-Cartier divisor numerically equivalent to  $D$, we may assume that $D=\sum_{i=1}^r a_iG_i$ is an effective divisor, where $G_i$ are prime divisors and $a_i\in \RR_{> 0}$.  For $m\in \ZZ_{>0}$, write $mD=N_m+\sum_{i=1}^r\sigma_{G_i}(mD)_{\ZZ}G_i$. Then 
$|mD|=|N_m|+\sum_{i=1}^r\sigma_{G_i}(mD)_{\ZZ}G_i$ where $|N_m|$ has no codimension one components in its base locus.

There exists a birational morphism $\phi_m:X_m\rightarrow X$ such that $X_m$ is normal and is a resolution of indeterminancy of the rational map determined by $|N_m|$ on $X$. Thus 
$\phi_m^*(mD)=M_m+\sum_{i=1}^r\sigma_{G_i}(mD)_{\ZZ}\overline G_i+F_m$ where 
$M_m$ and  $F_m$ are effective, $F_m$ has exceptional support for $\phi_m$, $\overline G_i$ is the proper transform of $G_i$ on $X_m$ and 
$|\phi_m^*(mD)|=|M_m|+\sum_{i=1}^r\sigma_{G_i}(mD)_{\ZZ}\overline G_i+F_m$ where 
$|M_m|$ is base point free. Thus $M_m$ is a nef integral Cartier divisor on $X_m$. 

Set 
$D_m=\sum_{i=1}^r\frac{\sigma_{G_i}(mL)_{\ZZ}}{m}\overline G_i+\frac{F_m}{m}$, so that $D_m$ is an effective $\RR$-Cartier divisor on $X_m$.  We have that
$\frac{1}{m}M_m\le \langle D\rangle$ in $L^{d-1}(\mathcal X)$ so that
$\pi_X(\frac{1}{m}M_m)\le \pi_X\langle D\rangle$ in $L^{d-1}(X)$. Now 
$$
\begin{array}{lll}
\pi_X(\frac{1}{m}M_m)&=&(\phi_{m})_*(\frac{1}{m}M_m)
=\frac{1}{m}(\phi_m)_*((\phi_m)^*(mD)-\sum_{i=1}^r\sigma_{G_i}(mD)_{\ZZ}\overline G_i-F_m)\\
&=&D-\sum_{i=1}^r \frac{\sigma_{G_i}(mD)_{\ZZ}}{m}G_i.
\end{array}
$$
Thus 
$$
P_{\sigma}(D)=D-\sum_{i=1}^r\sigma_{G_i}(mD)G_i =\lim_{m\rightarrow \infty}(D-\sum_{i=1}^r\frac{\sigma_{G_i}(mD)_{\ZZ}}{m}G_i)\le \pi_X(\langle D\rangle)
$$
in $L^{d-1}(X)$.

  Let $Z\in I(X)$ be  normal, with birational map $g:Z\rightarrow X$ and $N$ be a nef and big $\RR$-Cartier divisor on $Z$ and $E$ be an effective $\RR$-Cartier divisor on $Z$ such that $N+E=g^*(D)$. Let $\Gamma$ be a prime divisor on $Z$. Then
 $$
 \sigma_{\Gamma}(g^*(D))\le \sigma_{\Gamma}(N)+\mbox{ord}_{\Gamma}(E)=\mbox{ord}_{\Gamma}(E).
 $$
 Thus $N_{\sigma}(g^*(D))\le E$ and so $N\le P_{\sigma}(g^*(D))$.
 
 Let $\tilde \Gamma$ be a prime divisor on $X$ and let $\Gamma$ be the proper transform of $\tilde \Gamma$ on $Z$. Then $\sigma_{\Gamma}(g^*(D))=\sigma_{\tilde\Gamma}(D)$ so that $\pi_X(N)\le P_{\sigma}(D)$ in $WN^1(X)$.
 Thus $\pi_X(\langle D \rangle)\le P_{\sigma}(D)$ in $L^{d-1}(X)$. 
 \end{proof}

Let $X$ be a projective variety and $L_1,\ldots, L_{d-1}\in N^1(X)$. Suppose that $D$ is a big and movable $\RR$-Cartier divisor on $X$. Then the intersection product in $L^0(\mathcal X)=\RR$ is 
\begin{equation}\label{eq90}
\begin{array}{lll}
L_1\cdot \ldots \cdot L_{d-1}\cdot \langle D\rangle
&=&\rho_X(L_1)\cdot\ldots\cdot \rho_X(L_{d-1})\cdot \langle D\rangle
=L_1\cdot\ldots\cdot L_{d-1}\cdot \pi_X(\langle D\rangle)\\
&=&(L_1\cdot\ldots\cdot L_{d-1}\cdot P_{\sigma}(D))_X
=(L_1\cdot\ldots\cdot L_{d-1}\cdot D)_X
\end{array}
\end{equation}

\subsection{Volume of divisors}  Suppose that $X$ is a $d$-dimensional projective variety over a field $k$ and $D$ is a Cartier divisor on $X$. 
The volume of $D$ is (\cite[Definition 2.2.31]{L})
$$
{\rm vol}(D)=\limsup_{n\rightarrow \infty}\frac{\dim_k(\Gamma(X,\mathcal O_X(nD))}{n^d/d!}.
$$
 This lim sup is actually a limit. When $k$ is an algebraically closed field of characteristic zero, this is shown in Example 11.4.7 \cite{L}, as a consequence of   Fujita Approximation \cite{F2} (c.f. Theorem 10.35 \cite{L}). The limit is established in 
 \cite{LM} and  \cite{T} when $k$ is algebraically closed of arbitrary characteristic. A proof over an arbitrary field is given in  \cite[Theorem 10.7]{C1}.
 
  Since ${\rm vol}$ is a homogeneous function, it extends naturally to a function on 
$\QQ$-divisors, and it extends to a continuous function on $N^1(X)$ (\cite[Corollary 2.2.45]{L}), giving the volume of an arbitrary $\RR$-Cartier divisor. 

We have (\cite[Theorem 3.1]{BFJ}, \cite[Theorems 5.2 and 5.3]{C}) that for a pseudo effective $\RR$-Cartier divisor $D$ on $X$, 
\begin{equation}\label{eq44}
{\rm vol}(D)=\langle D^d\rangle.
\end{equation}

Further, we have by  \cite[Theorem 3.5]{FKL} that for an arbitrary $\RR$-Cartier divisor $D$ (or even an $\RR$-Weil divisor) on a normal variety $X$, that
$$
{\rm vol}(D)=\lim_{n\rightarrow \infty}\frac{\dim_k(\Gamma(X,\mathcal O_X(nD))}{n^d/d!}.
$$
Thus ${\rm vol}(D)={\rm vol}(P_{\sigma}(D))$ and so if $P_{\sigma}(D)$ is $\RR$-Cartier, then ${\rm vol}(D)=\langle P_{\sigma}(D)^d\rangle$.

\begin{Lemma} Suppose that $L$ is an $\RR$-Cartier divisor on a $d$-dimensional projective variety $X$ over a field $k$, $Y$ is a projective variety and $\phi:Y\rightarrow X$ is a generically finite morphism. Then
\begin{equation}\label{eq43}
{\rm vol}(\phi^*L)=\deg(Y/X)\,{\rm vol}(L).
\end{equation}
\end{Lemma}

\begin{proof} 
First assume that $L$ is a Cartier divisor. 
The sheaf $\phi_*\mathcal O_Y$ is a coherent sheaf of $\mathcal O_X$-modules. Let $R$ be the coordinate ring of $X$ with respect to some closed embedding of $X$ in a projective space. Then $R=\oplus_{i\ge 0}R_i$ is a standard graded domain over $R_0$, and $R_0$ a finite extension field of $k$. There exists a finitely generated graded $R$-module $M$ such that the sheafication $\tilde M$ of $M$ is isomorphic to $\phi_*\mathcal O_Y$ (by \cite[Proposition II.5.15 and Exercise II.5.9]{H} or \cite[Theorem 11.46]{AG}).
Let $S$ be the multiplicative set of nonzero homogeneous elements of $R$ and $\eta$ be the generic point of $X$. The ring $R_{(0)}$ is the set of homogeneous elements of degree 0 in the localization $S^{-1}R$ and the $R_{(0)}$-module $M_{(0)}$
 is the set of homogeneous elements of degree 0 in the localization $S^{-1}M$.
The function field of $X$ is
$k(X)=\mathcal O_{X,\eta} =R_{(0)}$ and $(\phi_*\mathcal O_Y)_{\eta}=M_{(0)}$ is a $k(X)$-vector space of rank $r=\deg(Y/X)$.
Let $f_1,\ldots,f_r\in M_{(0)}$ be a $k(X)$-basis. Write $f_i=\frac{z_i}{s_i}$ where $z_i \in M$ is homogeneous  of some degree $d_i$ and $s_i\in R$ is homogeneous of degree $d_i$.  Multiplication by $z_i$ induces a degree 0 graded $R$-module homomorphism $R(-d_i)\rightarrow M$ giving us  a degree 0 graded $R$-module  homomorphism
$\oplus_{i=1}^rR(-d_i)\rightarrow M$.  Let $K$ be the kernel of this homomorphism and $F$ be the cokernel. Let $\tilde K$ be the sheafification of $K$ and $\tilde F$ be the sheafification of $F$. We have a short exact sequence of coherent $\mathcal O_X$-modules
$0\rightarrow \tilde K\rightarrow \oplus_{i=1}^r\mathcal O_X(d_i)\rightarrow \pi_*\mathcal O_Y\rightarrow \tilde F\rightarrow 0$. Localizing at the generic point, we see that $\tilde K_{\eta}=0$ and $\tilde F_{\eta}=0$ so that the supports of $\tilde K$ and $\tilde F$ have dimension less than $\dim X$, and thus $K=0$ since it is a submodule of a torsion free $R$-module.  Tensoring the short exact sequence 
$0\rightarrow \oplus_{i=1}^r\mathcal O_X(d_i)\rightarrow \pi_*\mathcal O_Y\rightarrow \tilde F\rightarrow 0$ with $L^n$, we see that 
$$
{\rm vol}(\phi^*L)=\lim_{n\rightarrow \infty}\frac{\dim_k\Gamma(Y,\phi^*L^n)}{n^d/d!}=\lim_{n\rightarrow\infty}\frac{
\dim_k(\oplus_{n=1}^r\Gamma(X,\mathcal O_X(d_i)\otimes L^n))}{n^d/d!}=\deg(Y/X)\,{\rm vol}(L).
$$ 
Since volume is homogeneous, (\ref{eq43}) is valid for $\QQ$-Cartier divisors, and since volume is continuous on $N^1(X)$ and $N^1(Y)$, (\ref{eq43}) is valid for $\RR$-Cartier divisors.

\end{proof}

\section{A theorem on volumes}
In this section we generalize \cite[Theorem 4.2]{C2}. The proof given here is a variation of the one given in \cite{C2}, using the theory of divisorial Zariski decomposition of $\RR$-Weil divisors on normal varieties of \cite{FKL}.
Let $X$ be a $d$-dimensional normal projective variety over a field $k$. Suppose that $D$ is a big $\RR$-Weil divisor  on $X$. Let $E$ be a codimension one prime divisor on $X$. In \cite[Lemma 4.1]{FKL}    the function $\sigma_E$ of Subsection \ref{Subsecsigma} is generalized to give the following definition (\cite[Lemma 4.1]{FKL})
$$
\sigma_E(D)=\lim_{m\rightarrow\infty}\min \frac{1}{m}\{\mbox{mult}_ED'\mid D'\sim_{\ZZ} mD, D'\ge 0\}.
$$
Suppose that $D$ is a big $\RR$-Weil divisor and $E_1,\ldots,E_r$ are distinct prime divisors on $X$. Then by \cite[Lemma 4.1]{FKL}, for all $m\in \NN$,
\begin{equation}\label{eq70} 
\Gamma(X,\mathcal O_X(mD))=\Gamma(X,\mathcal O_X(mD-\sum_{i=1}^rm\sigma_{E_i}(D)E_i)).
\end{equation}





We now recall the method of \cite{LM} to compute volumes of graded linear series on $X$, as extended in \cite{C2} to arbitrary fields. We restrict to the situation of our immediate interest; that is,  $D$ is a big $\RR$-Weil divisor and $H$ is an ample Cartier divisor on $X$ such that $D\le H$. 

Suppose that $p\in X$ is a nonsingular closed point and 
\begin{equation}\label{eqGR2}
X=Y_0\supset Y_1\supset \cdots \supset Y_d=\{p\}
\end{equation}
is a flag; that is, the $Y_i$ are subvarieties of $X$ of dimension $d-i$ such that there is a regular system of parameters $b_1,\ldots,b_d$ in $\mathcal O_{X,p}$ such that  $b_1=\cdots=b_i=0$ are local equations of $Y_i$ in $X$ for $1\le i\le d$. 

The flag determines a valuation $\nu$ on the function field $k(X)$ of $X$  as follows.  We have a sequence of natural surjections of regular local rings
\begin{equation}\label{eqGR3} \mathcal O_{X,p}=
\mathcal O_{Y_0,p}\overset{\sigma_1}{\rightarrow}
\mathcal O_{Y_1,p}=\mathcal O_{Y_0,p}/(b_1)\overset{\sigma_2}{\rightarrow}
\cdots \overset{\sigma_{d-1}}{\rightarrow} \mathcal O_{Y_{d-1},p}=\mathcal O_{Y_{d-2},p}/(b_{d-1}).
\end{equation}
Define a rank $d$ discrete valuation $\nu$ on $k(X)$   by prescribing for $s\in \mathcal O_{X,p}$,
$$
\nu(s)=({\rm ord}_{Y_1}(s),{\rm ord}_{Y_2}(s_1),\cdots,{\rm ord}_{Y_d}(s_{d-1}))\in (\ZZ^d)_{\rm lex}
$$
where 
$$
s_1=\sigma_1\left(\frac{s}{b_1^{{\rm ord}_{Y_1}(s)}}\right),
s_2=\sigma_2\left(\frac{s_1}{b_2^{{\rm ord}_{Y_2}(s_1)}}\right),\ldots,
s_{d-1}=\sigma_{d-1}\left(\frac{s_{d-2}}{b_{d-1}^{{\rm ord}_{Y_{d-1}}(s_{d-2})}}\right).
$$

Let $g=0$ be a local equation of $H$ at $p$. For $m\in \NN$, define
$$
\Phi_{mD}:\Gamma(X,\mathcal O_X(mD))\rightarrow \NN^d
$$
by $\Phi_{mD}(f)=\nu(fg^m)$.  The Okounkov body $\Delta(D)$ of $D$ is the closure of the set 
$$
\cup_{m\in \NN}\frac{1}{m}\Phi_{mD}(\Gamma(X,\mathcal O_X(mD)))
$$
in $\RR^d$. 
$\Delta(D)$  is a compact and convex set by \cite[Lemma 1.10]{LM} or the proof of \cite[Theorem 8.1]{C1}.

By the proof of \cite[Theorem 8.1]{C1} and of \cite[Lemma 5.4]{C3} we see that
\begin{equation}\label{GR4}
{\rm Vol}(D)=\lim_{m\rightarrow \infty}\frac{\dim_k\Gamma(X,\mathcal O_X(mD))}{m^d/d!}
=d![\mathcal O_{X,p}/m_p:k]{\rm Vol}(\Delta(D)).
\end{equation}


The following proposition is proven with the assumption that the ground field $k$ is perfect in i) implies ii) of Theorem B in \cite{FKL}. The assumption that $k$ is perfect is required in their proof as they use 
\cite{T}, which  proves that a Fujita approximation exists 
to compute the volume of a Cartier divisor when the ground field is perfect. The theorem of \cite{dJ} is used in \cite{FKL} to conclude that a separable alteration exists if the ground field $k$ is perfect.

\begin{Proposition}\label{Prop1} Suppose that $X$ is a normal projective variety over a field $k$
 and $D_1,D_2$ are big $\RR$-Weil divisors on   $X$ such that   $D_1\le D_2$ and ${\rm Vol}(D_1)={\rm Vol}(D_2)$. 
 Then
$$
\Gamma(X,\mathcal O_X(nD_1))=\Gamma(X,\mathcal O_X(nD_2))
$$
for all $n\in \NN$.
\end{Proposition}

\begin{proof} Write $D_2=D_1+\sum_{i=1}^r a_iE_i$  where the $E_i$ are prime divisors on $X$ and 
 $a_i\in \RR_{>0}$ for all $i$. Let $H$ be an ample Cartier divisor on $X$ such that $D_2\le H$. 

 For each $i$ with $1\le i\le r$ choose a flag (\ref{eqGR2})
 with $Y_1=E_i$ and $p$ a point such that $p\in X$ is a nonsingular closed point of $X$ and $E_i$ and  $p\not\in E_j$ for $j\ne i$.  Let 
$\pi_1:\RR^d\rightarrow \RR$ be the projection onto the first factor.   For $f\in \Gamma(X,\mathcal O_X(mD_j))$,
$$
\frac{1}{m}\mbox{ord}_{E_i}(fg^m)=\frac{1}{m}\mbox{ord}_{E_i}((f)+mD_j)+\mbox{ord}_{E_i}(H-D_j).
$$
Thus
$$
\pi_1^{-1}(\sigma_{E_i}(D_j)+\mbox{ord}_{E_i}(H-D_j))\cap \Delta(D_j)\ne \emptyset
$$
 and
 $$
 \pi_1^{-1}(a)\cap \Delta(D_j)=\emptyset\mbox{ if }a<\sigma_{E_i}(D_j)+\mbox{ord}_{E_i}(H-D_j).
 $$
  
  Further, $\Delta(D_1)\subset \Delta(D_2)$ and ${\rm Vol}(D_1)={\rm Vol}(D_2)$, so $\Delta(D_1)=\Delta(D_2)$ by Lemma \cite[Lemma 3.2]{C2}.
  Thus
$$
\sigma_{E_i}(D_1)+\mbox{ord}_{E_i}(H-D_1)=\sigma_{E_i}(D_2)+\mbox{ord}_{E_i}(H-D_2)
$$
for $1\le i\le r$. We  obtain that
$$
D_2-\sum_{i=1}^r\sigma_{E_i}(D_2)E_i=D_1-\sum_{i=1}^r\sigma_{E_i}(D_1)E_i.
$$
By (\ref{eq70}), for all $m\ge 0$,
$$
\Gamma(X,\mathcal O_X(mD_1))=\Gamma(X,\mathcal O_X(mD_2)).
$$
\end{proof}

\begin{Lemma}\label{Lemma2}
Suppose that $X$ is a nonsingular projective variety and $D_1\le D_2$ are big $\RR$-divisors on $X$. Then the following are equivalent
\begin{enumerate}
\item[1)] ${\rm vol}(D_1)={\rm vol}(D_2)$
\item[2)] $\Gamma(X,\mathcal O_X(nD_1))=\Gamma(X,\mathcal O_X(nD_2))$ for all $n\in \NN$
\item[3)]  $P_{\sigma}(D_1)=P_{\sigma}(D_2)$.
\end{enumerate}
\end{Lemma}

\begin{proof} 1) implies 2) is Proposition \ref{Prop1}. We now assume 2) holds and prove 3). Then 
$|nD_2|=|nD_1|+n(D_2-D_1)$ for all $n\ge 0$. Thus
$$
\sigma_{\Gamma}(D_2)=\sigma_{\Gamma}(D_1)+\mbox{ord}_{\Gamma}(D_2-D_1),
$$
and so
$$
\begin{array}{lll}
P_{\sigma}(D_2)&=&D_2-N_{\sigma}(D_2)=D_1+(D_2-D_1)-(N_{\sigma}(D_1)+D_2-D_1)\\
&=& D_1-N_{\sigma}(D_1)=P_{\sigma}(D_1).
\end{array}
$$
Finally, we prove 3) implies 1). Suppose that $P_{\sigma}(D_1)=P_{\sigma}(D_2)$. Then
$$
{\rm vol}(D_1)={\rm vol}(P_{\sigma}(D_1))={\rm vol}(P_{\sigma}(D_2))={\rm vol}(D_2)
$$
by (\ref{eq70}).
\end{proof}

\section{The Augmented Base Locus}


Let $X$ be a normal variety over a field. Let $D$ be a big $\RR$-Cartier divisor on $X$.  The augmented base locus $B_+(D)$ is defined in \cite[Definition 1.2]{ELM} and extended to $\RR$-Weil divisors in \cite[Definition 5.1]{FKL}. $B_+^{\rm div}(D)$ is defined to be the divisorial part of $B_+(D)$. It is shown in \cite[Proposition 1.4]{ELM} that if $D_1$ and $D_2$ are big $\RR$-Cartier divisors and $D_1\equiv D_2$ then $B_+(D_1)=B_+(D_2)$. In \cite[Lemma 5.3]{FKL}, it is shown that if $A$ is an ample $\RR$-Cartier divisor on $X$, then 
\begin{equation}\label{eq61}
B_+^{\rm div}(D)=\mbox{Supp}(N_{\sigma}(D-\epsilon A))
\end{equation}
for all sufficiently small positive $\epsilon$.

The following Lemma is  $i)$ equivalent to $ii)$  of  \cite[Theorem B]{FKL}, in the case that $X$ is nonsingular, over an arbitrary field. We use Lemma \ref{Lemma2} to remove the assumption in \cite[Theorem B]{FKL} that the ground field is perfect.   

\begin{Lemma}\label{Lemma60} Let $X$ be a nonsingular projective variety over a field. Let $D$ be a big $\RR$-divisor on $X$ and $E$ be an effective $\RR$-divisor. Then ${\rm vol}(D+E)={\rm vol}(D)$ if and only if $\mbox{Supp}(E)\subset B_+^{\rm div}(D)$.
\end{Lemma}

\begin{proof} Suppose that ${\rm vol}(D+E)={\rm vol}(D)$. 

Let $D'$ be an $\RR$-divisor such that $D'\equiv D$. Then ${\rm vol}(D'+E)={\rm vol}(D')$. 
Lemma \ref{Lemma2} implies 
$\Gamma(X,\mathcal O_X(nD'))=\Gamma(X,\mathcal O_X(nD'+sE))$ for all $n>0$ and $0\le s\le n$. Thus $\Gamma(X,\mathcal O_X(nD'))=\Gamma(X,\mathcal O_X(nD'+rE))$ for all $n>0$ and $r\ge 0$ by \cite[Lemma III.1.8, Corollary III.1.9]{N} or \cite[Lemma 4.1]{FKL}.  Let $A$ be an ample $\RR$-divisor on $X$ and suppose that $F$ is an irreducible component of $E$ and $F\not\subset \mbox{Supp}(N_{\sigma}(D-\epsilon A))$ for $\epsilon$ sufficiently small.  By \cite[Lemma 4.9]{FKL}, there exists $m>0$ such that
$$
mD+F=(\frac{1}{2}m\epsilon A+F)+(\frac{1}{2}m\epsilon A+mP_{\sigma}(D-\epsilon A))+mN_{\sigma}(D-\epsilon A)
$$
is numerically equivalent to an effective divisor $G$ that does not contain $F$ in its support. Let $D'=\frac{1}{m}(G-F)\equiv D$. Then for $r$ sufficiently large,
$$
\dim_k\Gamma(X,\mathcal O_X(mD'+rE))\ge \dim_k\Gamma(X,\mathcal O_X(mD'+F))>\dim_k\Gamma(X,\mathcal O_X(mD')),
$$
giving a contradiction, and so by (\ref{eq61}), $\mbox{Supp}(E)\subset B_+^{\rm div}(D)$.

Now suppose that $\mbox{Supp}(E)\subset B_{+}^{\rm div}(D)$.  Let $A$ be an ample $\RR$-divisor on $X$. By (\ref{eq61}), we have that $\mbox{Supp}(E)\subset \mbox{Supp}(N_{\sigma}(D-\epsilon A))$ for all sufficiently small positive $\epsilon$. By \cite[Lemma 4.13]{FKL}, we have that ${\rm vol}(D+E-\epsilon A)={\rm vol}(D-\epsilon A)$ for all sufficiently small $\epsilon>0$.
Thus ${\rm vol}(D+E)={\rm vol}(D)$ by continuity of volume of $\RR$-divisors. 

\end{proof}


\section{The Minkowski equality}\label{SecMink}

In this section,  we modify the proof sketched in \cite{LX2} of  \cite[Proposition 3.7]{LX2} to be valid over an arbitrary field. Characteristic zero is required in the proof in \cite{LX2}  as the existence of resolution of singularities is assumed and an argument using the theory of multiplier ideals is used, which  requires characteristic zero as it relies on both resolution of singularities and Kodaira vanishing.

\begin{Proposition}\label{Prop3} Let $X$ be a nonsingular  projective $d$-dimensional variety over a field $k$. Suppose that $L$ is a big $\RR$-divisor on $X$, and $P$ and $N$ are $\RR$-divisors on $X$  such that  $L\equiv P+N$ where ${\rm vol}(L)={\rm vol}(P)$ and $N$ is pseudo effective. Then $P_{\sigma}(P)\equiv P_{\sigma}(L)$.
\end{Proposition}

\begin{proof}  Write $N=P_{\sigma}(N)+N_{\sigma}(N)$. 


 Since $L$ and $P$ are big $\RR$-Cartier divisors, by superadditivity and positivity of intersection products,
$$
\begin{array}{lll}
{\rm vol}(L)&=&\langle L^d\rangle
\ge\langle L^{d-1}\cdot P\rangle+\langle L^{d-1}\cdot N\rangle\\
&=& \langle(P+N)^{d-1}\cdot P\rangle + \langle L^{d-1}\cdot N\rangle\\
& \ge& \langle P^d\rangle +\langle L^{d-1}\cdot N\rangle
= {\rm vol}(P)+\langle L^{d-1}\cdot N\rangle.
\end{array}
$$
Thus $\langle L^{d-1}\cdot N\rangle =0$. Let $A$ be an ample Cartier divisor on $X$. There exists  a small real multiple $\overline A$ of $A$ such that $B:=L-\overline A$ is a big $\RR$-Cartier divisor. 
$$
0=\langle (\overline A+B)^{d-1}\cdot N\rangle \ge \langle \overline A^{d-1}\cdot N\rangle=\langle \overline A^{d-1}\cdot P_{\sigma}(N)+N_{\sigma}(N)\rangle \ge 
\langle \overline A^{d-1}\cdot P_{\sigma}(N)\rangle =\overline A^{d-1}\cdot \langle P_{\sigma}(N)\rangle
$$
by superadditivity and Lemma \ref{Lemma34}. 

By Lemma \ref{Lemma31}, $P_{\sigma}(N)+\epsilon \overline A$ is big and movable, so by (\ref{eq90}),
$$
\overline A^{d-1}\cdot \langle P_{\sigma}(N)+\epsilon \overline A\rangle=\overline A^{d-1}\cdot(P_{\sigma}(N)+\epsilon \overline A),
$$
so 
$$
\overline A^{d-1}\cdot \langle P_{\sigma}(N)\rangle =\lim_{\epsilon \rightarrow 0} \overline A^{d-1}\cdot \langle P_{\sigma}(N)+\epsilon \overline A\rangle=\overline A^{d-1}\cdot P_{\sigma}(N).
$$

Thus 
\begin{equation}\label{eq6}
(A^{d-1}\cdot P_{\sigma}(N))_X=0
\end{equation}
and so  $P_{\sigma}(N)\equiv 0$ by Lemma \ref{Lemma7}.
Thus $N\equiv N_{\sigma}(N)$.  Thus, replacing $P$ with the numerically equivalent divisor $P+P_{\sigma}(N)$,
we may assume that $N$ is effective. By Lemma \ref{Lemma2}, we have that 
$$
P_{\sigma}(P)=P_{\sigma}(P+N)\equiv P_{\sigma}(L).
$$
\end{proof}

\begin{Lemma}\label{Lemma10}  Let $X$ be a nonsingular $d$-dimensional projective variety over a field $k$. Suppose that $L_1$ and  $L_2$ are big $\RR$-divisors on $X$. Set $s$ to be the largest real number $s$ such that $L_1-sL_2$ is pseudo effective. Then
\begin{equation}\label{eq11}
s^d\le \frac{{\rm vol}(L_1)}{{\rm vol}(L_2)}
\end{equation}
and if  equality holds in (\ref{eq11}), then  $P_{\sigma}(L_1)\equiv sP_{\sigma}(L_2)$.
\end{Lemma}

\begin{proof} The pseudo effective cone is closed, so $s$ is well defined. We have $L_1\equiv sL_2+\gamma$ where $\gamma$ is pseudo effective. Thus ${\rm vol}(L_1)\ge {\rm vol}(sL_2)=s^d{\rm vol}(L_2)$. If this is an equality, then $sP_{\sigma}(L_2)\equiv P_{\sigma}(L_1)$ by Proposition \ref{Prop3}.
\end{proof}

Let $X$ be a projective variety over a field $k$. An alteration $\phi:Y\rightarrow X$ is a proper and dominant morphism such that $Y$ is a nonsingular  projective variety and $[k(Y):k(X)]<\infty$. If $X$ is normal and $D$ is a pseudo effective $\RR$-Cartier divisor on $X$, then by \cite[Lemma 4.12]{FKL},
\begin{equation}\label{eqNew20}
\phi_*N_{\sigma}(\phi^*D)=\deg(Y/X)N_{\sigma}(D).
\end{equation}

 It is proven in \cite{dJ} that for such $X$, an alteration always exists (although it may be that $k(Y)$ is not separable over $k(X)$ if $k$ is not perfect).

\begin{Lemma}\label{Lemma21} Suppose that $X$ is a   projective variety over a field $k$, $\phi:Y\rightarrow X$ is an alteration  and $L_1, L_2$ are pseudo effective $\RR$-Cartier divisors on $X$. 
Suppose that $s\in \RR_{>0}$. Then $\phi^*(L_1)-sP_{\sigma}(\phi^*(L_2))$ is pseudo effective if and only if $P_{\sigma}(\phi^*(L_1))-sP_{\sigma}(\phi^*(L_2))$ is pseudo effective.
\end{Lemma}

\begin{proof}  Certainly if $P_{\sigma}(\phi^*L_1)-sP_{\sigma}(\phi^*L_2)$ is pseudo effective then $\phi^*(L_1)-sP_{\sigma}(\phi^*L_2)$ is pseudo effective. Suppose $\phi^*(L_1)-sP_{\sigma}(\phi^*(L_2))$ is pseudo effective. Then there exists a pseudo effective  $\RR$-divisor $\gamma$ on $Y$ such that 
$$
P_{\sigma}(\phi^*L_1)+N_{\sigma}(\phi^*L_1)=\phi^*L_1\equiv  sP_{\sigma}(\phi^* L_2)+\gamma
=(sP_{\sigma}(\phi^*L_2)+P_{\sigma}(\gamma))+N_{\sigma}(\gamma).
$$
The effective $\RR$-divisor $N_{\sigma}(\gamma)$ has the property that  $\phi^*(L_1)-N_{\sigma}(\gamma)$ is movable by Lemma \ref{Lemma31}, so 
$N_{\sigma}(\gamma))\ge N_{\sigma}(\phi^*L_1)$ by \cite[Proposition III.1.14]{N}. Thus $P_{\sigma}(\phi^*L_1)-sP_{\sigma}(\phi^*L_2)$ is pseudo effective.

\end{proof}

\begin{Lemma}\label{Lemma22} 
Let $X$ be a  $d$-dimensional projective variety over a field $k$. Suppose that $L_1$ and  $L_2$ are big and movable $\RR$-Cartier divisors on $X$. Let $s$  be the largest real number  such that $L_1-sL_2$ is pseudo effective. Then
\begin{equation}\label{eq23}
s^d\le \frac{{\rm vol}(L_1)}{{\rm vol}(L_2)}
\end{equation}
and if  equality holds in (\ref{eq23}), then  $L_1$ and $L_2$ are proportional in $N^1(X)$.
\end{Lemma}

\begin{proof} Let $\phi:Y\rightarrow X$ be an alteration. 

Let $L$ be a big and movable $\RR$-Cartier divisor on $X$. Let $\Gamma\subset Y$ be a prime divisor which is not exceptional for $\phi$. Let $\tilde \Gamma$ be the codimension one subvariety of $X$ which is the  support of $\phi_*\Gamma$. 
Since $L$ is movable, there exist effective $\RR$-Cartier divisors $D_i$ on $X$ such that $\lim_{i\rightarrow \infty}D_i=L$ in $N^1(X)$ and $\tilde\Gamma\not\subset\mbox{Supp}(D_i)$ for all $i$. We thus have that $\phi^*(L)=\lim_{i\rightarrow\infty}\phi^*(D_i)$ in $N^1(Y)$ and $\Gamma\not\subset\mbox{Supp}(\phi^*(D_i))$ for all $i$, so that $\sigma_{\Gamma}(\phi^*(D_i))=0$ for all $i$. Thus $\sigma_{\Gamma}(\phi^*(L))=0$ since $\sigma_{\Gamma}$ is continuous on the big cone of $Y$.
Thus $N_{\sigma}(\phi^*L)$ has exceptional support for $\phi$ and thus $\phi_*(P_{\sigma}(\phi^*L))=\phi_*(\phi^*L)=\deg(Y/X)L$ by (\ref{eq41}).


Let $s_Y$ be the largest real number such that $P_{\sigma}(\phi^*L_1)-s_YP_{\sigma}(\phi^*L_2))$ is pseudo effective. Then $s_Y\ge s$  since $\phi^*L_1-s\phi^*L_2$ is pseudo effective and  by Lemma \ref{Lemma21}, and so
$$
s^d\le s_Y^d\le \frac{{\rm vol}(\phi^*L_1)}{{\rm vol}(\phi^*L_2)}=\frac{{\rm vol}(L_1)}{{\rm vol}(L_2)}
$$
by Lemma \ref{Lemma10}  and (\ref{eq43}).

If $s^d=\frac{{\rm vol}(L_1)}{{\rm vol}(L_2)}$, then $P_{\sigma}(\phi^*(L_1))=sP_{\sigma}(\phi^*(L_2))$ in $N^1(Y)$ by Lemma \ref{Lemma10}, and so 
$$
\deg(Y/X)(L_1-sL_2)=\phi_*(\phi^*(L_1)-s\phi^*(L_2))=\phi_*(P_{\sigma}(\phi^*(L_1))-s\phi_*(P_{\sigma}(\phi^*(L_2))=0
$$
 in $N_{d-1}(X)$,
 so that $0=L_1-sL_2$ in $N^1(X)$ by Lemma \ref{Lemma55}. 
 
\end{proof}

The following proposition is proven over an algebraically closed field of characteristic zero in \cite[Proposition 3.3]{LX2}.

\begin{Proposition}\label{Prop13}
Suppose that $X$ is a   projective $d$-dimensional variety over a field $k$ and $L_1,L_2$
 are big and moveable $\RR$-Cartier divisors on $X$. Then 
$$
\langle L_1^{d-1} \rangle\cdot L_2 \ge{\rm vol}(L_1)^{\frac{d-1}{d}}{\rm vol}(L_2)^{\frac{1}{d}}
$$
with equality if and only if $L_1$ and $L_2$ are proportional in $N^1(X)$.
\end{Proposition}

\begin{proof} 
Let $f:\overline X\rightarrow X$ be the normalization of $X$. Since $\overline X$ has no exceptional divisors for $f$, $f^*L_1$ and $f^*L_2$ are movable. We have that
$\langle f^*L_1^{d-1}\rangle \cdot f^*L_2 =\langle L_1^{d-1}\rangle \cdot L_2$ and $\rm{vol}(f^*L_i)={\rm vol}(L_i)$ for $i=1,2$. Further, 
$f^*:N^1(X)\rightarrow N^1(\overline X)$ is an injection, so $L_1$ and $L_2$ are proportional in $N^1(X)$ if and only if $f^*L_1$ and $f^*L_2$ are proportional in $N^1(\overline X)$. We may thus replace $X$ with its normalization $\overline X$, and so we can can assume for the remainder of the proof that $X$ is normal. 


We construct birational morphisms $\psi_m:Y_m\rightarrow X$ with 
numerically effective $\RR$-Cartier divisors $A_{i,m}$ and  effective $\RR$-Cartier divisors $E_{i,m}$  on $Y_m$ such that $A_{i,m}=\psi_m^*(L_i)-E_{i,m}$ and $\langle L_i\rangle =\lim_{m\rightarrow \infty}A_{i,m}$ in $L^{d-1}(\mathcal X)$ for $i=1,2$. We have that $\pi_X(A_{i,m})=\psi_{m,*}(A_{i,m})$ comes arbitrarily closed to $\pi_X(\langle L_j\rangle)=P_{\sigma}(L_j)=L_j$ in $L^{d-1}(X)$ by Lemma \ref{Lemma200}.

Let $s_L$ be the largest number such that $L_1-s_LL_2$ is pseudo effective and 
let $s_m$ be the largest number such that $A_{1,m}-s_mA_{2,m}$ is pseudo effective. 

We will now show that 
given $\epsilon>0$, there exists a positive integer $m_0$ such that $m>m_0$ implies $s_m<s_L+\epsilon$. Since ${\rm Psef}(X)$ is closed, there exists $\delta>0$ such that the open ball $B_{\delta}(L_1-(s_L+\epsilon)L_2)$ in $N^1(X)$ of radius $\delta$ centered at $L_1-(s_L+\epsilon)L_2$ is disjoint from ${\rm Psef}(X)$.  There exists $m_0$ such that $m\ge m_0$ implies
$\psi_{m*}(A_{1,m})\in B_{\frac{\delta}{2}}(L_1)$ and $\psi_{m*}(A_{2,m})\in B_{\frac{\delta}{(s_L+\epsilon)2}}(L_2)$. Thus
$\psi_{m*}(A_{1,m}-(s_L+\epsilon)A_{2,m})\not\in {\rm Psef}(X)$ for $m\ge m_0$ so that $s_m<s_L+\epsilon$.

By the Khovanski Teissier inequalities for nef and big divisors (\cite[Theorem 2.15]{BFJ} in characteristic zero, \cite[Corollary 6.3]{C}), 
\begin{equation}\label{eq14}
(A_{1,m}^{d-1}\cdot A_{2,m})^{\frac{d}{d-1}}\ge {\rm vol}(A_{1,m}){\rm vol}(A_{2,m})^{\frac{1}{d-1}}
\end{equation}
for all $m$. By Proposition \ref{Prop35}, taking limits as $m\rightarrow \infty$, we have
$$
\langle L_1^{d-1}\cdot L_2\rangle \ge {\rm vol}(L_1)^{\frac{d-1}{d}}{\rm vol}(L_2)^{\frac{1}{d}}.
$$
Now for each $m$, we have
$$
A_{1,m}^{d-1}\cdot \psi_m^*(L_2)=A_{1,m}^{d-1}\cdot (A_{2,m}+E_{2,m})\ge A_{1,m}^{d-1}\cdot A_{2,m}
$$
since $E_{2,m}$ is effective and $A_{1,m}$ is nef. Taking limits as $m\rightarrow \infty$, we have 
$\langle L_1^{d-1}\rangle \cdot L_2\ge \langle L_1^{d-1}\cdot L_2\rangle$. Thus
\begin{equation}\label{eq15}
\langle L_1^{d-1}\rangle\cdot L_2\ge \langle L_1^{d-1}\cdot L_2\rangle \ge {\rm vol}(L_1)^{\frac{d-1}{d}}{\rm vol}(L_2)^{\frac{1}{d}}.
\end{equation}
The Diskant inequality for big and nef divisors, \cite[Theorem 6.9]{C}, \cite[Theorem F]{BFJ} implies 
$$
(A_{1,m}^{d-1}\cdot A_{2,m})^{\frac{d}{d-1}}-{\rm vol}(A_{1,m}){\rm vol}(A_{2,m})^{\frac{1}{d-1}}
\ge ((A_{1,m}^{d-1}\cdot A_{2,m})^{\frac{1}{d-1}}-s_m{\rm vol}(A_{2,m})^{\frac{1}{d-1}})^d.
$$
We have that 
$(A_{1,m}^{d-1}\cdot A_{2,m})^{\frac{1}{d-1}}-s_m{\rm vol}(A_{2,m})^{\frac{1}{d-1}}\ge 0$ since 
$s_m^d\le\frac{{\rm vol}(A_{1,m})}{{\rm vol}(A_{2,m})}$ by Lemma \ref{Lemma22} and by (\ref{eq14}).

We have that 
$$
\begin{array}{lll}
\left[(A_{1,m}^{d-1}\cdot A_{2,m})^{\frac{d}{d-1}}-{\rm vol}(A_{1,m}){\rm vol}(A_{2,m})^{\frac{1}{d-1}}\right]^{\frac{1}{d}}
&\ge& (A_{1,m}^{d-1}\cdot A_{2,m})^{\frac{1}{d-1}}-s_m{\rm vol}(A_{2,m})^{\frac{1}{d-1}}\\
&\ge & (A_{1,m}^{d-1}\cdot A_{2,m})^{\frac{1}{d-1}}-(s_L+\epsilon){\rm vol}(A_{2,m})^{\frac{1}{d-1}}\end{array}
$$
for $m\ge m_0$.
Taking the limit as $m\rightarrow\infty$, we have 
\begin{equation}\label{eq16}
\langle L_1^{d-1} \cdot L_2\rangle ^{\frac{d}{d-1}}-{\rm vol}(L_1){\rm vol}(L_2)^{\frac{1}{d-1}}
\ge [\langle L_1^{d-1}\cdot L_2\rangle^{\frac{1}{d-1}}-s_L{\rm vol}(L_2)^{\frac{1}{d-1}}]^d.
\end{equation}

If $(\langle L_1^{d-1}\rangle \cdot L_2)^{\frac{d}{d-1}}={\rm vol}(L_1){\rm vol}(L_2)^{\frac{1}{d-1}}$ then
$\langle L_1^{d-1}\rangle \cdot L_2=\langle L_1^{d-1}\cdot L_2\rangle$ by (\ref{eq15}) and 
$(\langle L_1^{d-1}\rangle\cdot L_2)^{\frac{1}{d-1}}=s_L{\rm vol}(L_2)^{\frac{1}{d-1}}$, so that $s_L^d=\frac{{\rm vol}(L_1)}{{\rm vol}(L_2)}$ and thus $L_1$ and $L_2$ are proportional in $N^1(X)$ by Lemma \ref{Lemma22}.

Suppose $L_1$ and $L_2$ are proportional in $N^1(X)$, so that $L_1\equiv s_LL_2$ and $s_L^d=\frac{{\rm vol}(L_1)}{{\rm vol}(L_2)}$. Then
$$
\langle L_1^{d-1}\rangle \cdot L_2=s_L^{d-1}\langle L_2^{d-1}\rangle \cdot L_2=s_L^{d-1}\langle L_2^d\rangle
=\frac{{\rm vol}(L_1)^{\frac{d-1}{d}}}{{\rm vol}(L_2)^{\frac{d-1}{d}}}{\rm vol}(L_2)={\rm vol}(L_1)^{\frac{d-1}{d}}{\rm vol}(L_2)^{\frac{1}{d}}
$$
where the second equality is by (\ref{eq40}).
\end{proof}


The proof of the following theorem is deduced from Proposition \ref{Prop13} by extracting an argument from \cite[Theorem 4.11]{LX1}. Over algebraically closed fields of characteristic zero, it is \cite[Proposition 3.7]{LX2}.

\begin{Theorem}\label{Theorem18}
Let $L_1$ and $L_2$ be big and moveable $\RR$-Cartier divisors on a   $d$-dimensional projective variety $X$ over a field $k$. Then
\begin{equation}\label{eq97}
{\rm vol}(L_1+L_2)^{\frac{1}{d}}\ge {\rm vol}(L_1)^{\frac{1}{d}}+{\rm vol}(L_2)^{\frac {1}{d}}
\end{equation}
with equality if and only if $L_1$ and $L_2$ are proportional in $N^1(X)$.
\end{Theorem}

\begin{proof} By Theorem \ref{Theorem17}, we have that 
$$
\frac{d}{dt}{\rm vol}(L_1+tL_2)=d\langle (L_1+tL_2)^{d-1}\rangle\cdot L_2
$$
for $t$ in a neighborhood of $[0,1]$. By Proposition \ref{Prop13}, 
$$
\langle (L_1+tL_2)^{d-1}\rangle\cdot L_2\ge {\rm vol}(L_1+tL_2)^{\frac{d-1}{d}}{\rm vol}(L_2)^{\frac {1}{d}}.
$$ 
Thus 
\begin{equation}\label{eq19}
\begin{array}{lll}
{\rm vol}(L_1+L_2)^{\frac{1}{d}}-{\rm vol}(L_1)^{\frac{1}{d}}&=&\int_0^1{\rm vol}(L_1+tL_2)^{\frac{1-d}{d}}\langle(L_1+tL_2)^{d-1}\rangle\cdot L_2dt\\
& \ge& \int_0^1{\rm vol}(L_1+tL_2)^{\frac{1-d}{d}}{\rm vol}(L_1+tL_2)^{\frac{d-1}{d}}{\rm vol}(L_2)^{\frac {1}{d}}dt\\
&=& \int_0^1{\rm vol}(L_2)^{\frac{1}{d}}dt={\rm vol}(L_2)^{\frac{1}{d}}.
\end{array}
\end{equation}

Since positive intersection products are continuous on big divisors, we have equality in (\ref{eq19}) if and only if
$$
\langle(L_1+tL_2)^{d-1}\rangle\cdot L_2={\rm vol}(L_1+tL_2)^{\frac{d-1}{d}}{\rm vol}(L_2)^{\frac{1}{d}}
$$
for $0\le t\le 1$. Thus if equality holds in (\ref{eq97}), then $L_1$ and $L_2$ are proportional in $N^1(X)$ by Proposition \ref{Prop13}.

Since ${\rm vol}$ is homogeneous, if $L_1$ and $L_2$ are proportional in $N^1(X)$, then equality holds in (\ref{eq97}).
\end{proof}

The following theorem is proven over algebraically closed fields of characteristic zero in \cite[Theorem 1.6]{LX2}.

\begin{Theorem}\label{Theorem20}  Let $X$ be a nonsingular $d$-dimensional projective variety over a field $k$. For any two big $\RR$-divisors $L_1$ and $L_2$ on $X$, 
$$
{\rm vol}(L_1+L_2)^{\frac{1}{d}}\ge {\rm vol}(L_1)^{\frac{1}{d}}+{\rm vol}(L_2)^{\frac {1}{d}}
$$
with equality if and only if $P_{\sigma}(L_1)$ and $P_{\sigma}(L_2)$ are proportional in $N^1(X)$.
\end{Theorem}

\begin{proof} 
We have ${\rm vol}(P_{\sigma}(L_i))={\rm vol}(L_i)$ for $i=1, 2$. Since $L_i=P_{\sigma}(L_i)+N_{\sigma}(L_i)$ for $i=1,2$ where  $P_{\sigma}(L_i)$ is pseudo effective and movable and $N_{\sigma}(L_i)$ is effective, we have by super additivity of positive intersection products of pseudo effective divisors  and Theorem \ref{Theorem18} that 
$$
{\rm vol}(L_1+L_2)^{\frac{1}{d}}\ge{\rm vol}(P_{\sigma}(L_1)+P_{\sigma}(L_2))^{\frac{1}{d}}
\ge {\rm vol}(P_{\sigma}(L_1))^{\frac{1}{d}}+{\rm vol}(P_{\sigma}(L_2))^{\frac{1}{d}}={\rm vol}(L_1)^{\frac{1}{d}}+{\rm vol}(L_2)^{\frac{1}{d}}.
$$
Thus if we have the equality ${\rm vol}(L_1+L_2)^{\frac{1}{d}}={\rm vol}(L_1)^{\frac{1}{d}}+{\rm vol}(L_2)^{\frac{1}{d}}$,
we have 
$$
{\rm vol}(P_{\sigma}(L_1)+P_{\sigma}(L_2))^{\frac{1}{d}}={\rm vol}(P_{\sigma}(L_1))^{\frac{1}{d}}+{\rm vol}(P_{\sigma}(L_2))^{\frac{1}{d}}.
$$
 Then $P_{\sigma}(L_1)$ and $P_{\sigma}(L_2)$ are proportional in $N^1(X)$ by Theorem \ref{Theorem18}.

Now suppose that $P_{\sigma}(L_1)$ and $P_{\sigma}(L_2)$ are proportional in $N^1(X)$. Then there exists $s\in \RR_{>0}$ such that $P_{\sigma}(L_2)\equiv sP_{\sigma}(L_1)$, so that  $B_+^{\rm div}(P_{\sigma}(L_1))=B_+^{\rm div}(P_{\sigma}(L_2))$. 
Since ${\rm vol}(L_i)={\rm vol}(P_{\sigma}(L_i))$ for $i=1,2$, we have that
$\mbox{Supp}(N_{\sigma}(L_1)), \mbox{Supp}(N_{\sigma}(L_2))\subset B_+^{\rm div}(P_{\sigma}(L_1))$ by Lemma \ref{Lemma60}.
Thus $\mbox{Supp}(N_{\sigma}(L_1)+N_{\sigma}(L_2))\subset B_+^{\rm div}(P_{\sigma}(L_1))$, so that by Lemma \ref{Lemma60},
$$
{\rm vol}(L_1+L_2)={\rm vol}(P_{\sigma}(L_1)+sP_{\sigma}(L_1))=(1+s)^d{\rm vol}(P_{\sigma}(L_1)).
$$ 
Thus 
$$
{\rm vol}(L_1+L_2)^{\frac{1}{d}}=(1+s){\rm vol}(P_{\sigma}(L_1))^{\frac{1}{d}}={\rm vol}(L_1)^{\frac{1}{d}}+{\rm vol}(L_2)^{\frac{1}{d}}.
$$
\end{proof}


\section{Characterization of equality in the Minkowski inequality}

\begin{Theorem}\label{Theorem21}  Let $X$ be a normal $d$-dimensional projective variety. For any two big $\RR$-Cartier divisors $L_1$ and $L_2$ on $X$, 
$$
{\rm vol}(L_1+L_2)^{\frac{1}{d}}\ge {\rm vol}(L_1)^{\frac{1}{d}}+{\rm vol}(L_2)^{\frac {1}{d}}.
$$
If equality holds, then $P_{\sigma}(L_1)=sP_{\sigma}(L_2)$ in $N_{d-1}(X)$, where
$s=\left(\frac{{\rm vol}(L_1)}{{\rm vol}(L_2)}\right)^{\frac{1}{d}}$.

\end{Theorem}

\begin{proof} Here we use the extension of $\sigma$-decomposition to $\RR$-Weil divisors on a normal projective variety of \cite{FKL}.
Let $\phi:Y\rightarrow X$ be an alteration. We have that $\phi^*L_1$ and $\phi^*L_2$ are big $\RR$-Cartier divisors. By \cite[Lemma 4.12]{FKL}, for $i=1,2$,
$\phi_*N_{\sigma}(\phi^*L_i)=\deg(Y/X)\,N_{\sigma}(L_i)$. Since $\phi_*\phi^*L=\deg(Y/X)\,L$ by (\ref{eq41}), we have that 
$\phi_*P_{\sigma}(\phi^*L_i)=\deg(Y/X)\,P_{\sigma}(L_i)$. Now ${\rm vol}(\phi^*L_i)=\deg(Y/X)\, {\rm vol}(L_i)$ for $i=1,2$ and 
${\rm vol}(\phi^*L_1+\phi^*L_2)=\deg(Y/X)\,{\rm vol}(L_1+L_2)$ by (\ref{eq43}).

Thus the inequality of the statement of the theorem holds for $L_1$ and $L_2$ since it holds for $\phi^*L_1$ and $\phi^*L_2$ by Theorem \ref{Theorem20}.
Suppose that equality holds in the inequality. Then by Theorem \ref{Theorem20}, we have that there exists $s\in \RR_{>0}$ such that $P_{\sigma}(\phi^*L_1)=sP_{\sigma}(\phi^*L_2)$ in $N^1(Y)$. Thus $\phi_*P_{\sigma}(\phi^*L_1)=s\phi_*P_{\sigma}(\phi^*L_2)$
in $N_{d-1}(X)$, so that $P_{\sigma}(L_1)=sP_{\sigma}(L_2)$ in $N_{d-1}(X)$. Since volume is homogeneous and $P_{\sigma}(\phi^*L_1)$, $sP_{\sigma}(\phi^*L_2)$ are numerically equivalent $\RR$-Cartier divisors, 
$$
\frac{{\rm vol}(L_1)}{{\rm vol}(L_2)}=\frac{{\rm vol}(\phi^*L_1)}{{\rm vol}(\phi^*L_2)}
=\frac{{\rm vol}(P_{\sigma}(\phi^*L_1))}{{\rm vol}(P_{\sigma}(\phi^*L_2))}=s^d.
$$

\end{proof}

\begin{Theorem}\label{Theorem22}  Let $X$ be a  $d$-dimensional projective variety over a field $k$. For any two big $\RR$-Cartier divisors $L_1$ and $L_2$ on $X$, 
\begin{equation}\label{Neweq20}
{\rm vol}(L_1+L_2)^{\frac{1}{d}}\ge {\rm vol}(L_1)^{\frac{1}{d}}+{\rm vol}(L_2)^{\frac {1}{d}}
\end{equation}
with equality if and only if $\langle L_1\rangle $ and $\langle L_2\rangle$ are proportional in $L^{d-1}(\mathcal X)$. 
When this occurs, we have that $\langle L_1\rangle =s\langle L_2\rangle$  in $L^{d-1}(\mathcal X)$, where 
$s=\left(\frac{{\rm vol}(L_1)}{{\rm vol}(L_2)}\right)^{\frac{1}{d}}$.
\end{Theorem}

In the case that $D_1$ and $D_2$ are nef and big, this is proven in \cite[Theorem 2.15]{BFJ} (over an algebraically closed field of characteristic zero) and in \cite[Theorem 6.13]{C} (over an arbitrary field). In this case of nef divisors, the condition that 
$\langle L_1\rangle $ and $\langle L_2\rangle$ are proportional in $L^{d-1}(\mathcal X)$ is just that $D_1$ and $D_2$ are proportional in $N^1(X)$. 

Theorem \ref{Theorem22} is obtained in the case that $D_1$ and $D_2$ are big and movable and $k$ is an algebraically closed field of characteristic zero in \cite[Proposition 3.7]{LX2}. In this case the condition for equality is that $D_1$ and $D_2$ are proportional in $N^1(X)$.  Theorem \ref{Theorem22} is established in the case that $D_1$ and $D_2$ are big $\RR$-Cartier divisors and $X$ is nonsingular, over an algebraically closed field $k$ of characteristic zero in \cite[Theorem 1.6]{LX2}.  In this case, the condition for equality is that the positive parts of the $\sigma$ decompositions of $D_1$ and $D_2$ are proportional; that is, $P_{\sigma}(D_1)$ and $P_{\sigma}(D_2)$ are proportional in $N^1(X)$.

\begin{proof} Let $f:Y\rightarrow X\in I(X)$ with $Y$ normal. Then ${\rm vol}(f^*(L_1)+f^*(L_2))={\rm vol}(L_1+L_2)$ and ${\rm vol}(f^*L_j)={\rm vol}(L_j)$ for $j=1,2$ so that the inequality (\ref{Neweq20}) holds by Theorem \ref{Theorem21}. 

Suppose that equality holds in (\ref{Neweq20}). Let
$s=\left(\frac{{\rm vol}(L_2)}{{\rm vol}(L_1)}\right)^{\frac{1}{d}}$.
Then by Theorem \ref{Theorem21}, $P_{\sigma}(f^*L_1)=sP_{\sigma}(f^*L_2)$ in $N_{d-1}(Y)$.
Thus $\pi_Y(\langle L_1\rangle)=s\pi_Y(\langle L_2\rangle)$ by (\ref{eq91}). Since the normal $Y\in I(X)$ are cofinal in $I(X)$, we have that 
$\langle L_1\rangle =s\langle L_2\rangle$. 

Suppose that $\langle L_1\rangle=s\langle L_2\rangle$  in $L^{d-1}(\mathcal X)$ for some $s\in \RR_{>0}$. Then equality holds in (\ref{Neweq20}) by Proposition \ref{Prop35} and the fact that the positive intersection product is homogeneous. 
\end{proof}

\begin{Definition} Suppose that $X$ is a projective variety and $\alpha,\beta\in N^1(X)$. The slope of $\beta$ with respect to $\alpha$ is the smallest real number $s=s(\alpha,\beta)$ such that $\langle \alpha\rangle \ge s\langle \beta\rangle$.
\end{Definition}

Let $X$ be a projective variety and $f:Z\rightarrow X$ be a resolution of singularities.
Suppose that $L_1$ and $L_2$ are $\RR$-Cartier divisors on $X$. Let $\overline L_1=f^*(L_1)$ and $\overline L_2=f^*L_2$.
Suppose that $\phi:Y\rightarrow Z$ is a birational morphism of nonsingular projective varieties where $Y$ is nonsingular and $t\in \RR$. We will show that
\begin{equation}\label{eqZ1}
P_{\sigma}(\overline L_1)-tP_{\sigma}(\overline L_2)\mbox{ is pseudo effective if and only if }
P_{\sigma}(\phi^*\overline L_1)-tP_{\sigma}(\phi^*\overline L_2)\mbox{ is pseudo effective.}
\end{equation}
 The fact that $P_{\sigma}(\overline L_1)-tP_{\sigma}(\overline L_2)$ pseudo effective implies 
$P_{\sigma}(\phi^*\overline L_1)-tP_{\sigma}(\phi^*\overline L_2)$ pseudo effective follows from Lemma \ref{Lemma21}.
 If $P_{\sigma}(\phi^*\overline L_1)-tP_{\sigma}(\phi^*\overline L_2)$ is pseudo effective, then
$\phi_*(P_{\sigma}(\phi^*\overline L_1)-tP_{\sigma}(\phi^*\overline L_2))=P_{\sigma}(\overline L_1)-tP_{\sigma}(\overline L_2)$ is pseudo effective.

Let $s=s(L_1,L_2)$.  Since the $Y\rightarrow Z$ with $Y$ nonsingular are cofinal in $I(X)$, we have that 
\begin{equation}\label{eqZ2}
\begin{array}{l}
\mbox{$s$ is the largest positive number such that
$\pi_Z(\langle L_1\rangle -s\langle L_2\rangle)=P_{\sigma}(\overline L_1)-sP_{\sigma}(\overline L_2)$}\\
\mbox{is pseudo effective.}
\end{array}
\end{equation}

\begin{Proposition}\label{NewProp1} Suppose that $X$ is a variety over a field of characteristic zero and $L_1$, $L_2$ are big $\RR$-Cartier divisors on $X$. Let $s=s(L_1,L_2)$. Then 
\begin{equation}\label{Neweq1}
s^d\le\frac{\langle L_1^d\rangle}{\langle L_2^d\rangle}
\end{equation}
and we have equality in this equation if and only if 
$\langle L_1\rangle$ is proportional to $\langle L_2\rangle$ in $L^{d-1}(\mathcal X)$. If we have equality, then 
$\langle L_1\rangle=s\langle L_2\rangle$ in $L^{d-1}(\mathcal X)$.
\end{Proposition}

\begin{proof} Let $Y\in I(X)$ be nonsingular, with birational morphism $f:Y\rightarrow X$. Then by Lemma \ref{Lemma200},
$$
P_{\sigma}(f^*L_1)-sP_{\sigma}(f^*L_2)=
\pi_Y(\langle L_1\rangle) -s\langle L_2\rangle)\in{\rm Psef}(Y).
$$
Thus by Lemma \ref{Lemma10},
$$
s^d\le    \frac{{\rm vol}(P_{\sigma}(f^*L_1))}{{\rm vol}(P_{\sigma}(f^*L_2))}=
\frac{{\rm vol}(L_1)}{{\rm vol}(L_2)}=\frac{\langle L_1^d\rangle}{\langle L_2^d\rangle},
$$
and so the inequality (\ref{Neweq1}) holds. 

Suppose we have equality in (\ref{Neweq1}).
Let  $Y\in I(X)$ be nonsingular with morphism $f:Y\rightarrow X$. We have that 
$\pi_Y(\langle L_1\rangle)-s\pi_Y(\langle L_2\rangle)=
P_{\sigma}(f^*L_1)-sP_{\sigma}(f^*L_2)$ is pseudo effective and $s^d=\frac{{\rm vol}(P_{\sigma}(f^*L_1))}{{\rm vol}(P_{\sigma}(f^*L_2)}$, so we have that $P_{\sigma}(f^*L_1)= sP_{\sigma}(f^*L_2)$ in $N^1(Y)$ by (\ref{eqZ2}) and Lemma \ref{Lemma10}. Since the nonsingular $Y$ are cofinal in $I(X)$, we have that $\langle L_1\rangle=s\langle L_2\rangle$ by Lemma \ref{Lemma200} and (\ref{eq300}).

Suppose that $\langle L_1\rangle=t\langle L_2\rangle$ for some $t\in \RR_{>0}$. Then $s=t$ and by Proposition \ref{Prop35}, 
$$
\langle L_1^d\rangle=\langle L_1\rangle\cdot \ldots\cdot \langle L_1\rangle
=\langle sL_2\rangle\cdot \ldots\cdot \langle sL_2\rangle
=s^d \langle L_2\rangle\cdot \ldots\cdot \langle L_2\rangle=s^d\langle L_2^d\rangle .
$$
\end{proof}

\begin{Theorem}\label{PropNew60}(Diskant inequality for big divisors) 
Suppose that $X$ is a  projective $d$-dimensional variety over a field $k$ of characteristic zero and $L_1,L_2$
 are big  $\RR$-Cartier divisors on $X$.  Then 
 \begin{equation}\label{eq16*}
\langle L_1^{d-1} \cdot L_2\rangle ^{\frac{d}{d-1}}-{\rm vol}(L_1){\rm vol}(L_2)^{\frac{1}{d-1}}
\ge [\langle L_1^{d-1}\cdot L_2\rangle^{\frac{1}{d-1}}-s(L_1,L_2){\rm vol}(L_2)^{\frac{1}{d-1}}]^d.
\end{equation} 
 \end{Theorem}
 
 The Diskant inequality is proven for nef and big divisors in \cite[Theorem G]{BFJ} in characteristic zero and in  \cite[Theorem 6.9]{C} for nef and big divisors over an arbitrary field. In the case that $D_1$ and $D_2$ are nef and big, the condition  that $\langle D_1\rangle - s \langle D_2\rangle$ is pseudo effective in $L^{d-1}(\mathcal X)$ is that $D_1-sD_2$ is pseudo effective in $N^1(X)$. The Diskant inequality is proven when $D_1$ and $D_2$ are big and movable divisors and $X$ is a projective variety over an algebraically closed field of characteristic zero in \cite[Proposition 3.3, Remark 3.4]{LX2}.  Theorem \ref{PropNew60} is a consequence of   \cite[Theorem 3.6]{DF}.

\begin{proof}  Let $s=s(L_1,L_2)$. Let $f:Z\rightarrow X$ be a resolution of singularities. After replacing $L_i$ with $f^*L_i$ for $i=1,2$, we may assume that $X$ is nonsingular. 



We construct birational morphisms $\psi_m:Y_m\rightarrow X$ with 
numerically effective $\RR$-Cartier divisors $A_{i,m}$ and  effective $\RR$-Cartier divisors $E_{i,m}$  on $Y_m$ such that $A_{i,m}=\psi_m^*(L_i)-E_{i,m}$ and $\langle L_i\rangle =\lim_{m\rightarrow \infty}A_{i,m}$ in $L^{d-1}(\mathcal X)$ for $i=1,2$. We have that $\pi_X(A_{i,m})=\psi_{m,*}(A_{i,m})$ comes arbitrarily closed to $\pi_X(\langle L_j\rangle)=P_{\sigma}(L_j)$ in $L^{d-1}(X)$ by Lemma \ref{Lemma200}.

By (\ref{eqZ2}), $s$ is the largest number such that 
$P_{\sigma}(L_1)-sP_{\sigma}(L_2)$ is pseudo effective (in $N^1(X)$).  
Let $s_m$ be the largest number such that $A_{1,m}-s_mA_{2,m}$ is pseudo effective (in $N^1(Y_m)$). 

We will now show that 
given $\epsilon>0$, there exists a positive integer $m_0$ such that $m>m_0$ implies $s_m<s+\epsilon$. Since ${\rm Psef}(X)$ is closed, there exists $\delta>0$ such that the open ball $B_{\delta}(P_{\sigma}(L_1)-(s+\epsilon)P_{\sigma}(L_2))$ in $N^1(X)$ of radius $\delta$ centered at $P_{\sigma}(L_1)-(s+\epsilon)P_{\sigma}(L_2)$ is disjoint from ${\rm Psef}(X)$.  There exists $m_0$ such that $m\ge m_0$ implies
$\psi_{m*}(A_{1,m})\in B_{\frac{\delta}{2}}(P_{\sigma}(L_1))$ and $\psi_{m*}(A_{2,m})\in B_{\frac{\delta}{(s+\epsilon)2}}(P_{\sigma}(L_2))$. Thus
$\psi_{m*}(A_{1,m}-(s+\epsilon)A_{2,m})\not\in {\rm Psef}(X)$ for $m\ge m_0$ so that $s_m<s+\epsilon$.

By the Khovanski Teissier inequalities for nef and big divisors (\cite[Theorem 2.15]{BFJ} in characteristic zero, \cite[Corollary 6.3]{C}), 
\begin{equation}\label{eq14*}
(A_{1,m}^{d-1}\cdot A_{2,m})^{\frac{d}{d-1}}\ge {\rm vol}(A_{1,m}){\rm vol}(A_{2,m})^{\frac{1}{d-1}}
\end{equation}
for all $m$. By Proposition \ref{Prop35}, taking limits as $m\rightarrow \infty$, we have
\begin{equation}\label{eq20*}
\langle L_1^{d-1}\cdot L_2\rangle \ge {\rm vol}(L_1)^{\frac{d-1}{d}}{\rm vol}(L_2)^{\frac{1}{d}}.
\end{equation}

The Diskant inequality for big and nef divisors, \cite[Theorem 6.9]{C}, \cite[Theorem F]{BFJ} implies 
$$
(A_{1,m}^{d-1}\cdot A_{2,m})^{\frac{d}{d-1}}-{\rm vol}(A_{1,m}){\rm vol}(A_{2,m})^{\frac{1}{d-1}}
\ge ((A_{1,m}^{d-1}\cdot A_{2,m})^{\frac{1}{d-1}}-s_m{\rm vol}(A_{2,m})^{\frac{1}{d-1}})^d.
$$
We have that 
$(A_{1,m}^{d-1}\cdot A_{2,m})^{\frac{1}{d-1}}-s_m{\rm vol}(A_{2,m})^{\frac{1}{d-1}}\ge 0$ since 
$s_m^d\le\frac{{\rm vol}(A_{1,m})}{{\rm vol}(A_{2,m})}$ by Lemma \ref{Lemma22} and by (\ref{eq14*}).

We have that 
$$
\begin{array}{lll}
\left[(A_{1,m}^{d-1}\cdot A_{2,m})^{\frac{d}{d-1}}-{\rm vol}(A_{1,m}){\rm vol}(A_{2,m})^{\frac{1}{d-1}}\right]^{\frac{1}{d}}
&\ge& (A_{1,m}^{d-1}\cdot A_{2,m})^{\frac{1}{d-1}}-s_m{\rm vol}(A_{2,m})^{\frac{1}{d-1}}\\
&\ge & (A_{1,m}^{d-1}\cdot A_{2,m})^{\frac{1}{d-1}}-(s+\epsilon){\rm vol}(A_{2,m})^{\frac{1}{d-1}}\end{array}
$$
for $m\ge m_0$.
Taking the limit as $m\rightarrow\infty$, we have that (\ref{eq16*}) holds.


\end{proof}

\begin{Proposition}\label{Prop13*} 
Suppose that $X$ is a  projective $d$-dimensional variety over a field $k$ of characteristic zero and $L_1,L_2$
 are big  $\RR$-Cartier divisors on $X$. Then 
$$
\langle L_1^{d-1} \cdot L_2\rangle \ge{\rm vol}(L_1)^{\frac{d-1}{d}}{\rm vol}(L_2)^{\frac{1}{d}}.
$$
If equality holds, then $\langle L_1\rangle = s\langle L_2\rangle$ in $L^{d-1}(\mathcal X)$, where $s=s(L_1,L_2)=\left(\frac{{\rm vol}(L_2)}{{\rm vol}(L_1)}\right)^{\frac{1}{d}}$.
\end{Proposition}

\begin{proof} The inequality holds by (\ref{eq20*}). Let $s=s(L_1,L_2)$. By (\ref{eq16*}),
if $\langle L_1^{d-1} \cdot L_2\rangle^{\frac{d}{d-1}}={\rm vol}(L_1){\rm vol}(L_2)^{\frac{1}{d-1}}$ then
$\langle L_1^{d-1}\cdot L_2\rangle^{\frac{1}{d-1}}=s{\rm vol}(L_2)^{\frac{1}{d-1}}$, so that $s^d=\frac{{\rm vol}(L_1)}{{\rm vol}(L_2)}$ and thus 
$\langle L_1\rangle =s\langle L_2\rangle$ in $L^{d-1}(\mathcal X)$ by Proposition \ref{NewProp1}.


\end{proof}


Suppose that $X$ is a complete $d$-dimensional algebraic variety over a field $k$ and $D_1$, $D_2$ are pseudo effective $\RR$-Cartier divisors on $X$. We will write
$$
s_i=\langle D_1^i\cdot D_2^{d-i}\rangle\mbox{ for $0\le i\le d$}. 
$$

We have the following generalization of the Khovanskii-Teissier inequalities to positive intersection numbers.

\begin{Theorem} (Minkowski Inequalities)\label{Ineq} Suppose that $X$ is a complete algebraic variety of dimension $d$ over a field $k$ and $D_1$ and $D_2$ are pseudo effective $\RR$-Cartier divisors on $X$. Then
\begin{enumerate}
\item[1)] $s_i^2\ge s_{i+1}s_{i-1}$ for $1\le i\le d-1.$
\item[2)]  $s_is_{d-i}\ge s_0s_d$ for $1\le i\le d-1$.
\item[3)] $s_i^d\ge s_0^{d-i}s_d^i$ for $0\le i\le d$.
\item[4)] $\langle (D_1+D_2)^d\rangle \ge \langle D_1^d\rangle^{\frac{1}{d}}+\langle D_2^d\rangle^{\frac{1}{d}}$.
\end{enumerate}
\end{Theorem}

\begin{proof} Statements 1) - 3)  follow from the inequality of Theorem 6.6 \cite{C} (\cite[Theorem 2.15]{BFJ} in characteristic zero). Statement 4) follows from 3) and 
 the super additivity of the positive intersection  product.
\end{proof}

When $D_1$ and $D_2$ are nef, the inequalities of Theorem \ref{Ineq} are proven by Khovanskii  and Teissier \cite{T1}, \cite{T2}, \cite[Example 1.6.4]{L}.  In the case that $D_1$ and $D_2$ are nef, we have that $s_i=\langle D_1^i\cdot D_2^{d-i}\rangle=(D_1^i\cdot D_2^{d-i})$ are the ordinary intersection products.

We have the following characterization of equality in these inequalities.

\begin{Theorem} (Minkowski equalities)\label{Minkeq} Suppose that $X$ is a projective algebraic variety of dimension $d$ over a field $k$ of characteristic zero, and $D_1$ and $D_2$ are big $\RR$-Cartier divisors on $X$.  Then the following are equivalent:
\begin{enumerate}
\item[1)] $s_i^2= s_{i+1}s_{i-1}$ for $1\le i\le d-1.$
\item[2)]  $s_is_{d-i}= s_0s_d$ for $1\le i\le d-1$.
\item[3)] $s_i^d= s_0^{d-i}s_d^i$ for $0\le i\le d$.
\item[4)] $s_{d-1}^d=s_0s_d^{d-1}$.
\item[5)] $\langle (D_1+D_2)^d\rangle = \langle D_1^d\rangle^{\frac{1}{d}}+\langle D_2^d\rangle^{\frac{1}{d}}$.
\item[6)] $\langle D_1\rangle$ is proportional to $\langle D_2\rangle$ in $L^{d-1}(\mathcal X)$.
\end{enumerate}
\end{Theorem}

When $D_1$ and $D_2$ are nef and big, then Theorem \ref{Minkeq} is proven  in \cite[Theorem 2.15]{BFJ} when $k$ has characteristic zero and in  \cite[Theorem 6.13]{C} for arbitrary $k$. When $D_1$ and $D_2$ are nef and big, the condition
6) of Theorem \ref{Minkeq} is just that $D_1$ and $D_2$ are proportional in $N^1(X)$.

\begin{proof}
All the numbers $s_i$ are positive by Remark \ref{Remark50}.
Proposition \ref{Prop35} shows that 6) implies 1), 2), 3),  4) and 5).
 Theorem \ref{Theorem22} shows that 5) implies 6). 
Proposition \ref{Prop13*} shows that 4) implies 6). Since the condition of 3) is a subcase of the condition 4), we have that 3) implies 6).

Suppose that 2) holds. By the inequality 3) of Theorem \ref{Ineq} and the  equality 2), we have that
$$
s_i^ds_{d-i}^d\ge (s_0^{d-i}s_d^i)(s_0^is_d^{d-i})=(s_0s_d)^d=(s_is_{d-i})^d.
$$
Thus the inequalities 3) hold.

Suppose that the inequalities 1) hold. Then
$$
\frac{s_{d-1}}{s_0}=\frac{s_{d-1}}{s_{d-2}}\frac{s_{d-2}}{s_{d-3}}\cdots\frac{s_1}{s_0}=\left(\frac{s_d}{s_{d-1}}\right)^{d-1}
$$
so that 4) holds.
\end{proof}

\begin{Remark} The existence of resolutions of singularities is the only place where characteristic zero is used in the proof of Theorem \ref{Minkeq}. Thus the conclusions of Theorem \ref{Minkeq} are valid over an arbitrary field for varieties of dimension $d\le3$ by \cite{Ab2}, \cite{CP}.
\end{Remark}
Generalizing Teissier \cite{T1}, we 
 define the inradius of $\alpha$ with respect to $\beta$ as
$$
r(\alpha;\beta)=s(\alpha,\beta)
$$
and the outradius of $\alpha$ with respect to $\beta$ as
$$
R(\alpha;\beta)=\frac{1}{s(\beta,\alpha)}.
$$

\begin{Theorem}\label{TheoremG} Suppose that $X$ is a $d$-dimensional projective variety over a field $k$ of characteristic zero and $\alpha,\beta$ are big $\RR$-Cartier divisors on $X$.  Then
\begin{equation}\label{eq106}
\frac{s_{d-1}^{\frac{1}{d-1}}-(s_{d-1}^{\frac{d}{d-1}}-s_0^{\frac{1}{d-1}}s_d)^{\frac{1}{d}}}{s_0^{\frac{1}{d-1}}}
\le r(\alpha;\beta)\le \frac{s_d}{s_{d-1}}.
\end{equation}
\end{Theorem}

\begin{proof} Let $s=s(\alpha,\beta)=r(\alpha,\beta)$. Since $\langle \alpha\rangle \ge s\langle \beta\rangle$, we have that $\langle \alpha^d\rangle\ge s\langle \beta\cdot\alpha^{d-1}\rangle$ by Lemma \ref{Lemma36}. This gives us the upper bound. We also have that
\begin{equation}\label{eq110} 
\langle \alpha^{d-1}\cdot\beta\rangle^{\frac{1}{d-1}}-s\langle \beta^d\rangle^{\frac{1}{d-1}}\ge 0.
\end{equation}
 We  obtain the lower bound from Theorem \ref{PropNew60} (using  the inequality $s_{d-1}^d\ge s_0s_d^{d-1}$ to ensure that the bound is a positive real number).
\end{proof}

\begin{Theorem}\label{TheoremH} Suppose that $X$ is a $d$-dimensional projective variety over a field $k$ of characteristic zero and $\alpha,\beta$ are big $\RR$-Cartier divisors on $X$. Then
\begin{equation}\label{eq107}
\frac{s_{d-1}^{\frac{1}{d-1}}-(s_{d-1}^{\frac{d}{d-1}}-s_0^{\frac{1}{d-1}}s_d)^{\frac{1}{d}}}{s_0^{\frac{1}{d-1}}}
\le r(\alpha;\beta)\le \frac{s_d}{s_{d-1}}\le\frac{s_1}{s_0}\le R(\alpha;\beta)\le 
\frac{s_d^{\frac{1}{d-1}}}{s_1^{\frac{1}{d-1}}-(s_1^{\frac{d}{d-1}}-s_d^{\frac{1}{d-1}}s_0)^{\frac{1}{d}}}.
\end{equation}
\end{Theorem}

\begin{proof} By Theorem \ref{TheoremG}, we have that
$$
\frac{s_1^{\frac{1}{d-1}}-(s_1^{\frac{d}{d-1}}-s_d^{\frac{1}{d-1}}s_0)^{\frac{1}{d}}}{s_d^{\frac{1}{d-1}}}
\le s(\beta,\alpha)\le\frac{s_0}{s_1}.
$$
The theorem now follows from the fact that $R(\alpha,\beta)=\frac{1}{s(\beta,\alpha)}$ and Theorem \ref{Ineq}.
\end{proof}

This gives a solution to \cite[Problem B]{T1} for big $\RR$-Cartier divisors. The inequalities of Theorem \ref{TheoremH} are proven by Teissier in \cite[Corollary 3.2.1]{T1} for divisors on surfaces satisfying some conditions. 
 In the case that $D_1$ and $D_2$ are nef and big on a projective variety over a field of characteristic zero, Theorem \ref{TheoremH} follows from the Diskant inequality \cite[Theorem F]{BFJ}. In the case that $D_1$ and $D_2$ are nef and big on a projective variety over an arbitrary field, Theorem \ref{TheoremH} is proven in \cite[Theorem 6.11]{C}, as a consequence of the Diskant inequality \cite[Theorem 6.9]{C} for nef divisors.


\end{document}